\documentclass[a4paper,oneside,11pt]{article}

\usepackage{geometry}
\usepackage{fancyhdr}
\usepackage{lipsum}

\usepackage{amsmath,amsfonts,amscd,amssymb,amsthm}
\usepackage{longtable}
\usepackage[english]{babel}
\usepackage[utf8]{inputenc}
\usepackage[active]{srcltx}
\usepackage[T1]{fontenc}
\usepackage{graphicx}
\usepackage{enumitem}
\usepackage{bbm}
\usepackage{enumerate}   
\usepackage{mathtools}
\usepackage[colorlinks = true,
            linkcolor = blue,
            urlcolor  = magenta,
            citecolor = blue,
            anchorcolor = blue]{hyperref}
\usepackage{subfig}
\usepackage{MnSymbol}
\usepackage{stmaryrd}
\usepackage{nicefrac}

 \usepackage{rotating}

\usepackage{xcolor}
\usepackage{framed}

\colorlet{shadecolor}{blue!15}

\newtheorem{theorem}{Theorem}[section]

\newtheorem{corollary}[theorem]{Corollary}
\newtheorem{lemma}[theorem]{Lemma}
\newtheorem{proposition}[theorem]{Proposition}

\newtheorem{remark}[theorem]{Remark}

\numberwithin{equation}{section}

\newcommand{\calA}{\mathcal{A}}

\newcommand{\calC}{\mathcal{C}}

\newcommand{\calO}{\mathcal{O}}

\newcommand{\calV}{\mathcal{V}}

\newcommand{\calX}{\mathcal{X}}

\newcommand{\bbP}{\mathbb{P}}

\newcommand{\bbR}{\mathbb{R}}

\newcommand{\bbZ}{\mathbb{Z}}

\newcommand{\R}{\mathbb{R}}

\renewcommand{\P}{\mathbb{P}}

\newcommand{\OO}{\mathcal{O}}
\newcommand{\V}{\mathcal{V}}

\newcommand{\B}{\mathcal{B}}
\newcommand{\w}{\mathrm{w}}
\newcommand{\cross}{\mathrm{cross}}

\newcommand{\PivChain}{{\sf PivCh}}
\newcommand*{\email}[1]{E-mail: \href{mailto:#1}{#1} } 

\usepackage{abstract}

\begin{document}

 \title{Widths of crossings in Poisson Boolean percolation}
\author{Ioan Manolescu\footnote{Université de Fribourg, Switzerland.
    \email{ioan.manolescu@unifr.ch}}
  \qquad 
  Leonardo V.\ Santoro\footnote{EPFL,
    Switzerland. \email{leonardo.santoro@epfl.ch}}
    }

\maketitle

\begin{center}\small
\textbf{Abstract}

\bigskip

\begin{minipage}{.9\textwidth}

We answer the following question: 
if the occupied (or vacant) set of a planar Poisson Boolean percolation model does contain a crossing of an $n\times n$ square, 
how wide is this crossing? 
The answer depends on the whether we consider the critical, sub- or super-critical regime, and is different for the occupied and vacant sets. 

\bigskip
\footnotesize
\noindent \textbf{MSC2020 classes:} 60K35; 60G55
\end{minipage}
\end{center}


%


\section{Introduction}

Percolation is the branch of probability theory that investigates the geometry and connectivity properties of random media. Since its foundation in the 1950s  to model the diffusion of liquid in porous media \cite{Broadbent1957},  percolation theory has attracted great research interest and lead to significant discoveries, especially in two dimensions: see for instance,  Kesten’s determination of the critical threshold \cite{Kesten1980} and of the scaling relations \cite{Kesten1987},  Schramm’s introduction of the Schramm-Loewner evolution \cite{Lawler2008} and Smirnov’s proof of Cardy’s formula \cite{Smirnov2001}.  For an introduction and overview of percolation theory,  we recommend~\cite{Hammersley1984},~\cite{Grimmett1999} and~\cite{Bollobas2006}, among many other texts.

Bernoulli percolation on a symmetric grid lies at the foundation of all percolation models,  and embodies the archetypal setting to investigate  phase transitions and other phenomena emanating from statistical physics.  In the site-percolation version of this model,  each node of the grid is independently chosen to be black with probability $p$ and white with probability $1 - p$, and a random graph is obtained by removing the white vertices. 
It is well known that, as the parameter $p$ increases, the model undergoes a sharp phase transition at some critical parameter. 
At the point of phase transition, percolation is expected to exhibit a universal and conformally invariant scaling limit; unfortunately this was only proved in the particular case of site-percolation on the triangular lattice \cite{Smirnov2001}. 

Boolean (or \textit{continuum}) percolation first appeared in \cite{Gilbert1961} as an early mathematical model for wireless networks. In recent years, it has been studied in order to determine the theoretical bounds of information capacity and performance in such networks \cite{Dousse2006}.  See also \cite{Franceschetti2008} for a wider perspective on random networks.  In addition to this setting, continuum percolation has gained applications in other disciplines, including biology, geology, and physics, such as the study of porous materials and semiconductors. 
From a mathematical point of view, continuum percolation is particularly interesting as it is expected to exhibit the same features as discrete percolation, but enjoys additional symmetries, such as invariance under rotations. 
We direct the reader to~\cite{Meester1996} and the reference therein for more background.

\subsection{Framework}

Let $\eta$ be a Poisson point process on $\R^2$ with intensity $\lambda \cdot {\rm Leb}_{\mathbb R^2}$, where $\lambda > 0$ is a parameter of the model. 
Around each point in the support of $\eta$,  draw a disc of random radius, 
sampled independently for each point according to a fixed probability measure $\mu$ on $\R_+$. 
The set $\mathcal{O} = \OO_\eta \subset \R^2$ of points which are covered by at least one of the above discs is called the \textit{occupied set}, 
while its complement $\mathcal{V} = \calV_\eta := \R^2 \backslash \mathcal{O}$ is referred to as the \textit{vacant set}.
Write $\mathbb P_\lambda$ for the measure governing $\eta$ and the random sets $\mathcal O$ and $\mathcal V$. 

\smallskip

\begin{figure}[h!]
	\begin{center}
	\includegraphics[width=0.75\textwidth]{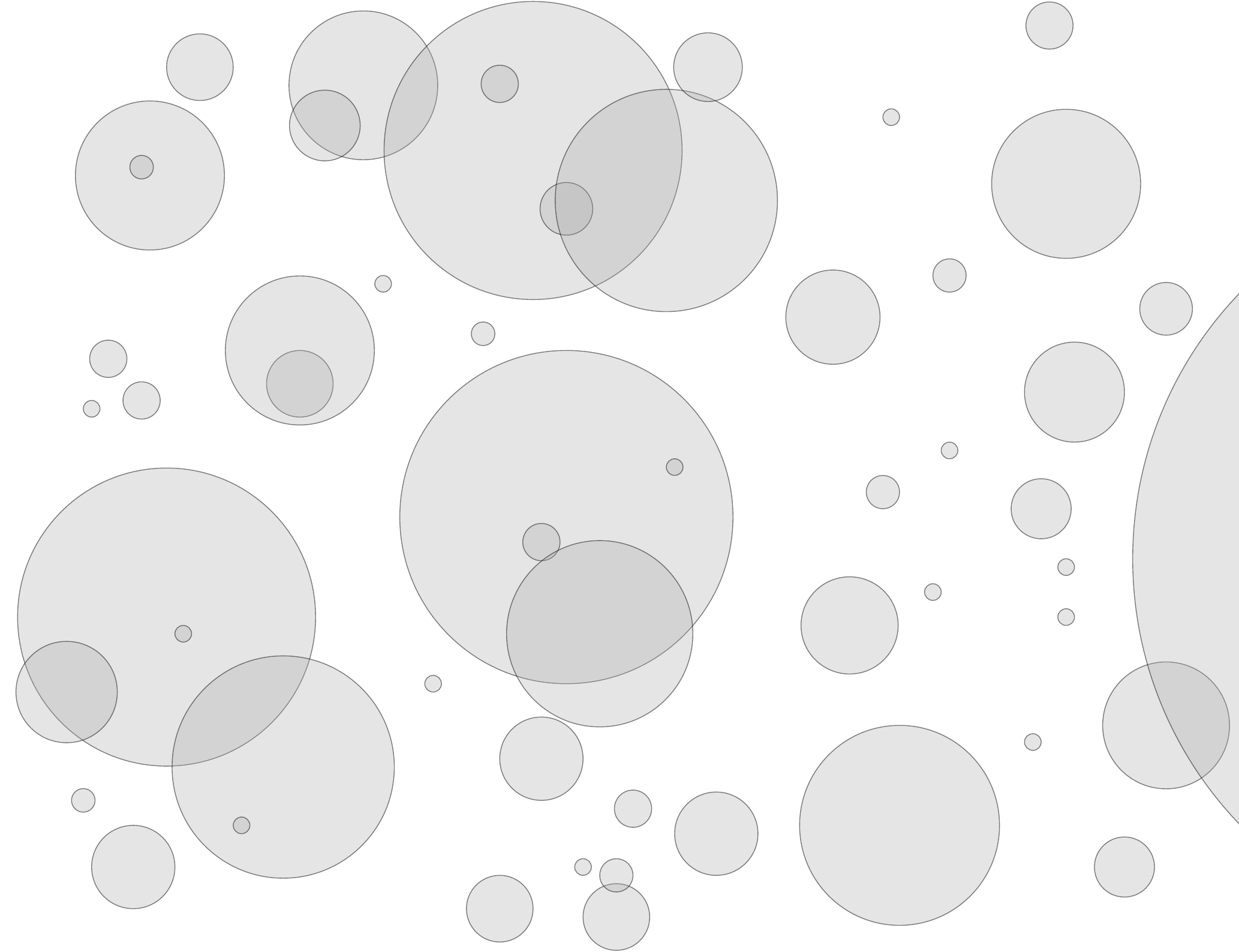}
	\caption{Continum percolation on $\R^2$.}
	\end{center}
\end{figure}

While sharing many features with site or bond Bernoulli percolation, the continuum model poses significant additional challenges: apart from being continuous rather than discrete, these come from its asymmetrical nature (the ``open" and ``closed" set have different properties) and potential long-range dependencies.  Nevertheless, similarly as for the classical Bernoulli case, the Boolean model undergoes a sharp phase transition as $\lambda$ increases.
Indeed,  set:
$$
\lambda_c := \sup\left\lbrace\lambda \geq  0 \::\: \P_\lambda\left(\,0 \overset{\mathcal{O}}{\leftrightsquigarrow} \infty \right) = 0 \right\rbrace,
 $$
 where $\,0 \overset{\mathcal{O}}{\leftrightsquigarrow} \infty $ is the event that the origin lies in an unbounded connected component of $\mathcal O$.  
Under \textit{minimal conditions}, it was recently shown by \cite{Ahlberg2018b} that $0<\lambda_c < \infty$ and:
\begin{itemize}
	\item[(i)] For $\lambda < \lambda_c$ (\textit{sub-critical phase})  the vacant set has a unique unbounded connected component,
 	and the probability of observing an occupied path from $0$ to distance $n$ decays exponentially fast in $n$.
	\item[(ii)] For $\lambda = \lambda_c$ (\textit{critical phase})
 	no unbounded connected component exists in either the occupied or the vacant set.
	Moreover the probability of observing either a vacant or occupied path from $0$ to distance $n$ decays polynomially fast in $n$.
	\item[(iii)] For $\lambda > \lambda_c$ (\textit{super-critical phase})  the occupied set has a unique unbounded component 
	and the probability of observing a vacant path from $0$ to distance $n$ decays exponentially fast in $n$.
\end{itemize}

%

\subsection{Results}

For simplicity, we limit our study to the particular setting where the radii are all equal to $1$ (i.e. when $\mu$ is the Dirac measure at $1$). 
The proof extends readily to the case where $\mu$ is supported on a compact of $(0,+\infty)$, 
and with additional work may include situations where $\mu$ has sufficiently fast decay towards $0$ and $\infty$. 
We do not investigate the optimal conditions for which the results remain valid. 

A \textit{crossing} of the square $[-n,n]^2$ is  a path contained in $[-n,n]^2$, 
connecting its left and right boundaries, namely $\{-n\}\times [-n,n]$ and $\{n\}\times [-n,n]$, respectively.  
A crossing is said to be occupied (resp. vacant) if it is entirely contained in $\mathcal O$ (resp. $\mathcal V$). 
In the following, we write $\cross(n)$ and $\cross^*(n)$  for the events that there exists 
an occupied and vacant horizontal crossing of $[-n,n]^2$, respectively.

The \textit{width} $w(\gamma)$ of an occupied crossing $\gamma$ is \textit{twice} the radius of the largest ball that can be transported along $\gamma$ without intersecting the vacant set. Alternatively, it may be viewed as twice the distance between $\gamma$ and $\mathcal V$:
\begin{equation}
	 2 \cdot\sup\{r \geq  0 \::\: B(\gamma(t),r) \subset \mathcal O \; \forall\; t\in [0,1]\} =  2\cdot\mathrm{dist}(\gamma, \V),
\end{equation}
where $\gamma$ is parameterized by $[0,1]$, and $B(x,r)$ denotes the euclidian ball of radius $r$ centered at $x \in \bbR^2$. 
A similar definition may be given for a vacant crossing, with the roles of $\mathcal O$ and $\mathcal V$ reversed. 

Define the \textit{maximal} occupied and vacant widths as:
\begin{equation}\label{eq:def:width}
\w_n := 2\cdot \sup\limits_{\gamma}  \mathrm{dist}(\gamma, \V) \qquad \text{ and } \qquad
\w^*_n := 2\cdot \sup\limits_{\gamma}  \mathrm{dist}(\gamma, \OO),
\end{equation}
where both supremums are taken over all horizontal crossings of $[-n,n]^2$.
Observe that when no occupied or vacant crossing exists, we have $\w_n = 0$ and $\w_n^* = 0$, respectively.

\begin{figure}[h!]
\centering
\subfloat[]{ \includegraphics[height=.22\textheight]{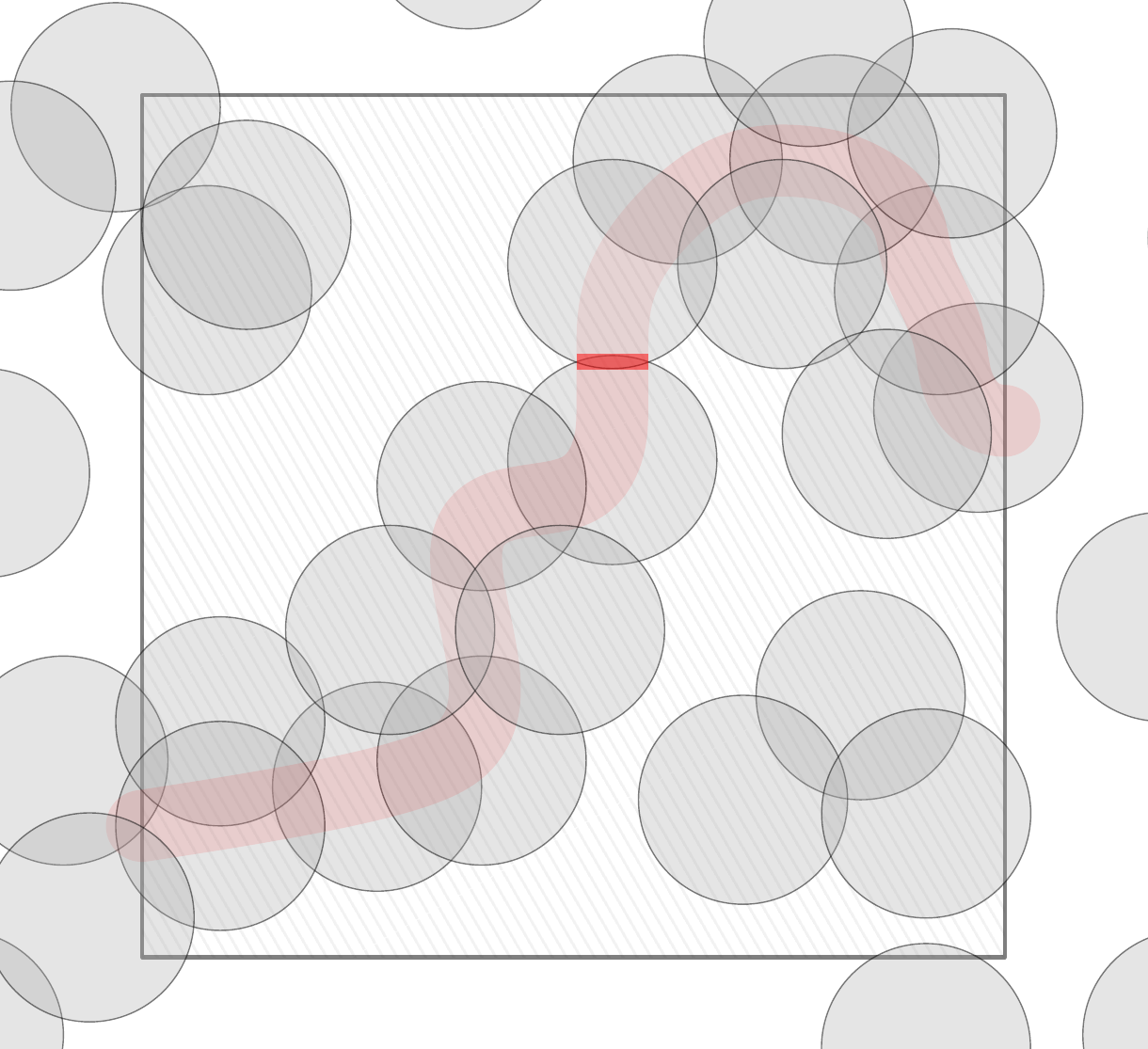}}
\hspace{1cm}
\subfloat[]{ \includegraphics[height=.22\textheight]{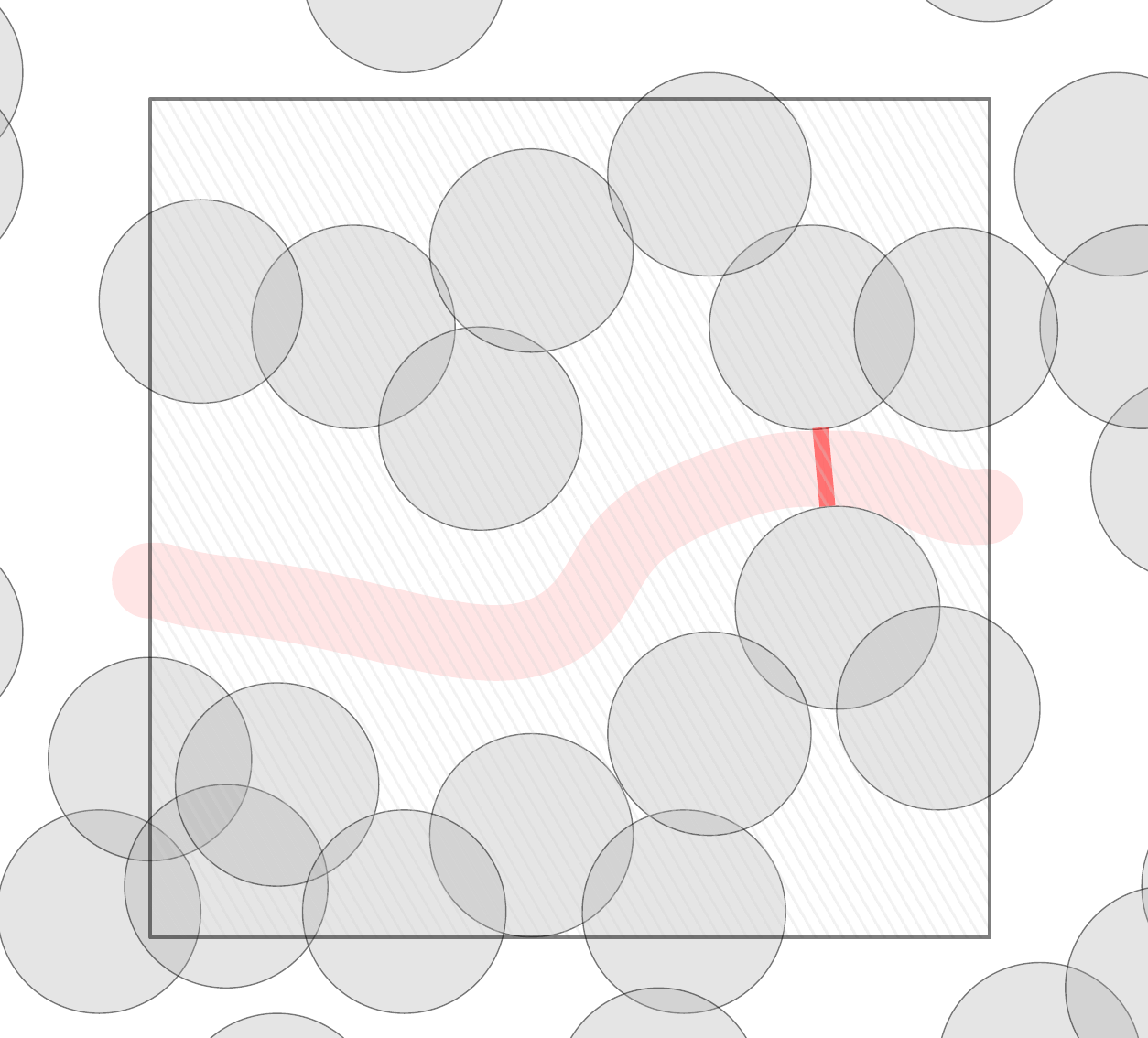}}
    \caption{The \textit{width} of an occupied (a) and vacant (b) crossing of an $n\times n$ square.}%
    \label{fig_widthOfCrossings}%
\end{figure}

Before stating our results, we define the four-arm probabilities. 
For $r \leq R$, define the four-arm event $A_4(r,R)$ as the existence of four disjoint paths $\gamma_1,\dots, \gamma_4$ 
in $\overline{B(0,R)}\setminus B(0,r)$, each starting on $\partial B(0,r)$ and ending on $\partial B(0,R)$, 
distributed in counter-clockwise order and with $\gamma_1,\gamma_3 \in \mathcal O$ and $\gamma_2,\gamma_4 \in \mathcal V$. 
Let: 
\begin{align*}
	\pi_4(r,R) = \mathbb P_{\lambda_c}[A_4(r,R)] \qquad \text{and} \qquad \pi_4(R)= \pi_4(1,R). 
\end{align*}

Our main result concerns the maximal widths of occupied and vacant crossings, when these are conditioned to exist. 
We formulate two theorems, respectively concerning the vacant and occupied cases.

\begin{theorem}[widths of vacant crossings]\label{thm:main_vacant}
	For any $\delta > 0$ and $\lambda > 0$, 
	there exist constants $0<c<C$, such that for large enough $n$:
	\begin{align}
		\P_{\lambda}\Big[\big| \w^*_n - 2\big(\tfrac{\lambda_c}{\lambda} - 1\big)\big| \leq \tfrac{C}{\lambda n^2\pi_4(n)} \;\Big|\; \cross^*(n) \Big] &\geq  1-\delta,
	    \qquad \text{ if } \lambda < \lambda_c,
	    \label{eq:main_subcritical_dual} \\
		\P_{\lambda_c}\Big[\tfrac{c}{n^2\pi_4(n)} \leq \w^*_n \leq \tfrac{C}{n^2\pi_4(n)} \;\Big|\; \cross^*(n) \Big]& \geq  1-\delta,
		\qquad\text{ if } \lambda = \lambda_c,
 		\label{eq:main_critical_dual}\\
		\P_{\lambda}\Big[\tfrac{c}{n} \leq \w^*_n \leq \tfrac{C}{n} \;\Big|\; \cross^*(n) \Big]&\geq  1-\delta,
		\qquad \text{ if } \lambda > \lambda_c.
		\label{eq:main_supercritical_dual}
	\end{align}
\end{theorem}

\begin{theorem}[widths of occupied crossings]\label{thm:main_occupied}
	For any $\delta > 0$ and $\lambda > 0$, 
	there exist constants $0<c<C$, such that for large enough $n$:
	\begin{align}
		\P_{\lambda}\Big[\tfrac{c}{ \sqrt n} \leq \w_n \leq \tfrac{C}{ \sqrt n} \;\Big|\; \cross(n) \Big] &\geq  1-\delta, 
		\qquad \text{ if } \lambda < \lambda_c,
		\label{eq:main_subcritical_primal}\\ 
		\P_{\lambda_c}\Big[\tfrac{c}{n\sqrt{\pi_4(n)}} \leq \w_n \leq \tfrac{C}{n\sqrt{\pi_4(n)}} \;\Big|\; \cross(n) \Big] &\geq  1-\delta,
		\qquad \text{ if } \lambda = \lambda_c, 		\label{eq:main_critical_primal} \\
	    \P_{\lambda}\Big[\w_n \geq 2 \sqrt{1 - \big(\tfrac{\lambda_c}{\lambda} + \tfrac{C}{\lambda n^2\pi_4(n)}\big)^2}\;\Big|\; \cross(n) \Big] &\geq  1-\delta,
		\qquad \text{ if } \lambda > \lambda_c.
		\label{eq:main_supercritical_primal}		
	\end{align}	
\end{theorem}

\smallskip

\begin{remark}\label{rem:degenerate_conditioning}
	The events in the conditioning in Theorems~\ref{thm:main_vacant} and~\ref{thm:main_occupied}  may be replaced by $\w^*_n > 0$ and $\w_n > 0$, respectively. 
	In~\eqref{eq:main_subcritical_dual},~\eqref{eq:main_critical_dual},~\eqref{eq:main_critical_primal} and~\eqref{eq:main_supercritical_primal}, 
	the conditioning has limited effect, 
	as the events have uniformly positive probability (even probability tending to $1$ in the first and last case).  
	However, in~\eqref{eq:main_supercritical_dual} and~\eqref{eq:main_subcritical_primal},
	the event in the conditioning is of exponentially small probability, and the resulting measure is highly degenerate. 
	In these cases, the vacant (and respectively occupied) clusters crossing the box have a specific structure
	described in detail in \cite{CamCC91,CamIof02,CamIofVel04}; 
	these works refer to Bernoulli percolation on the square lattice, but the statements and proofs adapt readily. 
	We will use these results in clearly stated forms, but without reproving them. 
\end{remark}

\paragraph{Organisation of the paper} 
Section~\ref{sec:preliminaries} contains certain background on the continuum percolation model. 
In particular we state a result concerning the near-critical regime, 
whose proof we only sketch, as it is very similar to existing arguments. 
Section~\ref{sec:key} contains an observation on two distinct increasing couplings of $\P_\lambda$ for $\lambda > 0$ that is the key to our arguments. 
In Section~\ref{sec:proofs_easy} we prove most of our two main results, Theorems~\ref{thm:main_vacant} and~\ref{thm:main_occupied},
using the observations of Section~\ref{sec:key}. 
Unfortunately, the upper bounds on $\w_n$ of Theorem~\ref{thm:main_occupied} are not accessible with this technique, 
and in Section~\ref{sec:proofs_pivotals} we prove these bounds using an alternative approach. 
Finally, in Section~\ref{sec:open_question}, we provide some related open questions.

\paragraph{Acknowledgements}
We thank Vincent Tassion who proposed this question during an open problem session 
at the ``Recent advances in loop models and height functions'' conference in 2019. 
The authors are grateful to Hugo Vanneuville for discussions on adapting the scaling relations to continuum percolation. 

This work started when the second author was a master student at the ETHZ, 
and the first author was visiting the FIM; we thank both institutions for their hospitality. 
The first author is supported by the Swiss NSF.

\section{Background and preliminaries}\label{sec:preliminaries}


\subsection{Positive association}

There is a natural partial ordering ``$\preceq"$ on space of possible realisations of $\eta$: 
we write $\omega \preceq \omega' $ if and only if $\OO_\omega \subset \OO_{\omega'}$.
An event $A$ is said to be \textit{increasing} if $\OO_\omega \in A $ implies that $\OO_{\omega'}\in A$ for all $\omega \preceq \omega'$. 
A useful property of increasing events is that they are positively correlated. Indeed, if $A_1$ and $A_2$ are both increasing events: 
\begin{equation}
	\label{eq:FGK}
    \P_\lambda[ A_1\cap A_2] \geq \P_\lambda[A_1] \cdot \P_\lambda[A_2].
\end{equation}
This result is known as FGK {inequality} and was proven by Roy in his doctorate thesis~\cite{Roy1990}. For a nice proof using discretization and martingale theory see \cite{Meester1996}.


\subsection{Russo's Formula}
Russo's differential formula controls how the probability of a monotone event varies under perturbations of the intensity parameter $\lambda$, assessing the variation in terms of \textit{pivotal} events. See \cite{Last2017} for a proof.

\begin{proposition}[Russo's formula]\label{prop : Russo's Formula}
	Let $A$ be an increasing event depending only on a \textit{bounded} subset of $\R^2$. Then, for every $\lambda>0$:
	\begin{align}\label{eq:russoP:der}
		\frac{{d}}{{d}\lambda}\P_{\lambda}[A] = \int\limits_{x\in \R^2} \P_\lambda \left[\mathrm{Piv}_{x}(A)\right] {d}x
	\end{align}
	where $\mathrm{Piv}_{x}(A) := \{\OO \notin A \} \cap \{\OO \cup \B_1(x) \in A\}$
	is the event that $x$ is \textit{pivotal} for $A$.
\end{proposition}


\subsection{Crossing probabilities and RSW theory}

Let $\mathrm{cross}(r,h)$ denote the event that there exists an occupied path inside the rectangle $[-r, r] \times [-h, h]$ between its left and right sides.
We write $\mathrm{cross}^*(r,h)$ for the corresponding event for the vacant set. 
Loosely speaking, the Russo-Seymour-Welsh (RSW) theory states that a lower bound on the crossing probability for a rectangle of aspect ratio $\rho$ 
implies a lower bound for a rectangle of larger aspect ratio $\rho'$. 

\begin{proposition}[RSW]\label{prop:RSW}
    For every $\rho,\,\rho' >0$ and $\varepsilon>0$ there exits $\varepsilon'= \varepsilon'(\rho,\rho')>0$ such that:
    \begin{align}\label{eq:RSW}
	    \P_\lambda[\mathrm{cross}(\rho n,n)] > \varepsilon \; \Longrightarrow  \P_\lambda[\mathrm{cross}(\rho'n,n)] > \varepsilon '.
    \end{align}
    for all $n \geq 1$. The same holds for $\mathrm{cross}^\star.$
\end{proposition}

 RSW bounds for continuum percolation were obtained separately for the occupied and vacant sets by Roy \cite{Roy1990} and Alexander \cite{Alexander1996} respectively, assuming in both cases heavy restrictions on the radii distribution, and later by Ahlberg et\ al.\ \cite{Ahlberg2018b} under minimal assumptions. A fundamental consequence of the RSW theorem (see again \cite{Ahlberg2018b})  concerns the abrupt change in crossing probabilities: 
\begin{itemize}
\item[(i)] For $\lambda < \lambda_c$ and all $\rho>0$, there exists $c > 0$ such that:
\begin{equation*}
	\P_\lambda[\mathrm{cross}(n,\rho n)] \leq e^{-cn}, \qquad \forall \:n\geq 1. .
\end{equation*}
\item[(ii)] At criticality, the \textit{box-crossing property }holds. That is, for every $\rho>0$ there exists $c=c(\rho)>0$ such that:
\begin{equation}\label{eq:box-crossing property}
c\leq \P_{\lambda_c}[\mathrm{cross}(\rho n,n)]\leq 1-c, \qquad \forall \:n\geq 1.
\end{equation}

\item[(iii)] For $\lambda > \lambda_c$ and all $\rho>0$, there exists $c > 0$ such that:
\begin{equation*}
	\P_\lambda[\mathrm{cross}(n,\rho n)] \geq 1 - e^{-cn}, \qquad \forall \:n\geq 1. 
\end{equation*}
\end{itemize}
The above results may be translated for the vacant set using the duality observation 
\begin{align} \label{eq:duality}
\P_\lambda[\mathrm{cross}(n,\rho n)] + \P_\lambda[\mathrm{cross}^*(\rho n,n)] = 1.
\end{align}
Indeed, a rectangle is a.s. either crossed horizontally by an occupied path, or vertically be a vacant one. 

Let us also give a corollary relating the crossing of slightly longer rectangles to that of squares. 

\begin{lemma}\label{lem:rRR}
	There exists $C > 0$ such that, for any $R \geq r \geq 1$, 
	\begin{align}\label{eq:rRR}
		\P_{\lambda_c}[\cross(R)] - \P_{\lambda_c}\big[\cross(R + r, R)\big]	\leq C \tfrac{r}{R}.
	\end{align}
\end{lemma}

\begin{proof}
	For simplicity, assume that $R$ is a multiple of $r$; the general case may be deduced by the monotonicity in $r$ of the left-hand side of~\eqref{eq:rRR}.
	Due to the inclusion of events, the difference in~\eqref{eq:rRR} may be written as 
	\begin{align}\label{eq:rRR2}
		\P_{\lambda_c}[\cross(R)] - \P_{\lambda_c}\big[\cross(R + r, R)\big]= \P_{\lambda_c}\big[\tilde \calC(R) \setminus {\calC}(R+r)\big],
	\end{align}
	where $\calC(k)$ is the event that there exists a horizontal occupied crossing of $[0,2k] \times [-R,R]$ 
	and $\tilde\calC(k)$ is the translation of this event by $2r$ to the right (see Figure~\ref{fig:calC}). 

	When $\tilde \calC(R) \setminus {\calC}(R+r)$ occurs, there exists a vertical vacant crossing of $[0,2r + 2R]\times[-R,R]$ 
	and a horizontal occupied crossing of $[2r,2r + 2R]\times[-R,R]$;
	the vertical crossing avoids the horizontal one by using the strip $[0,2r] \times [-R,R]$.
	Now, conditionally on this event, due to the RSW property (Proposition~\ref{prop:RSW}), 
	the horizontal crossing may be extended with positive probability into 
	a horizontal occupied crossing of $[0,2r + 4R]\times[-R,R]$.
	This step is not completely immediate, as it requires to prove a separation property, 
	by which the horizontal occupied crossing may be taken to end on  $\{2R\}\times[-R,R]$, far from the vertical vacant crossing.
	This type of property is classical and we will not detail its proof.

	\begin{figure}
    \centering
    	\includegraphics[height = 4cm, page =1]{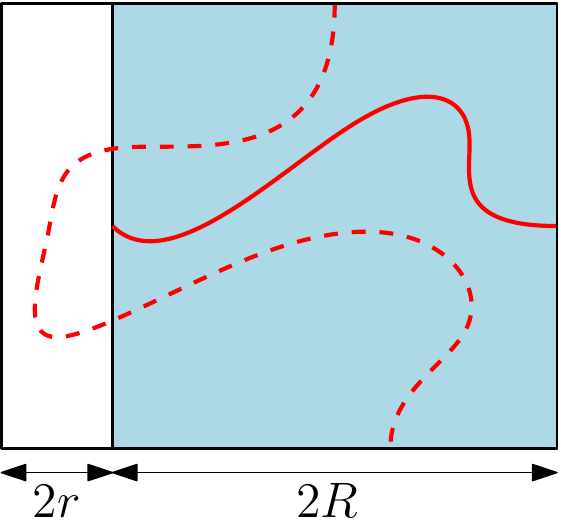}\hspace{0.05\textwidth}
    	\includegraphics[height = 4cm, page =1]{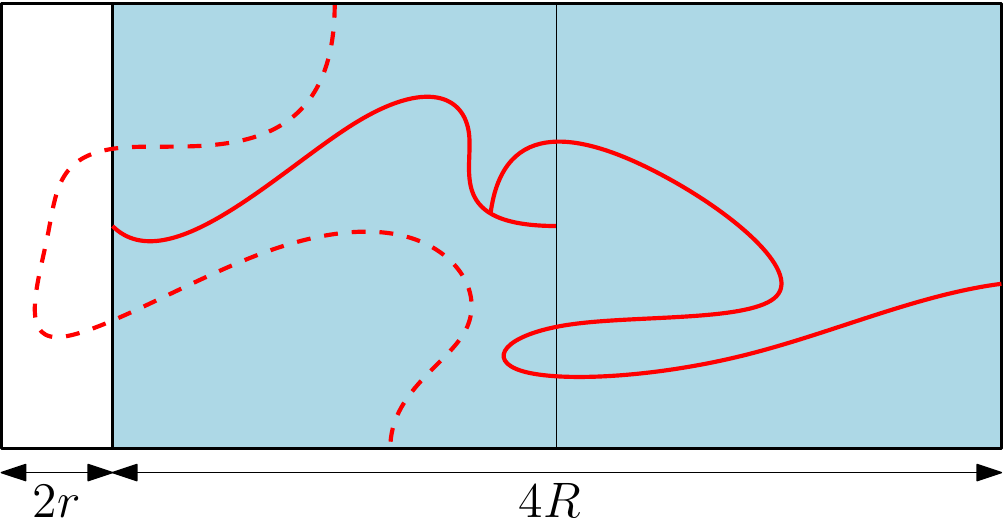}
    	\caption{{\bf Left:} A configuration in $\tilde \calC(R) \setminus {\calC}(R+r)$
    	contains a horizontal occupied crossing of the blue square $[2r,2r + 2R]\times[-R,R]$,
    	but no crossing of the slightly longer rectangle $[0,2r + 2R]\times[-R,R]$.
    	Occupied crossing are depicted by bold lines, vacant ones by dashed lines. 
    	{\bf Right:} The horizontal crossing of $[2r,2r + 2R]\times[-R,R]$ may be lengthened into one of $[2r,2r + 4R]\times[-R,R]$ at constant cost, due to the RSW theorem. 
    	This configuration belongs to  $\tilde \calC(R+kr) \setminus {\calC}(R+(k+1)r)$ for any $0 \leq k < R/r$.}
    	\label{fig:calC}
	\end{figure}
	
	 Observe that, when $[0,2r + 2R]\times[-R,R]$ contains a vertical vacant crossing, but $[2r,2r + 4R]\times[-R,R]$ contains a occupied horizontal crossing, 
	all events of the type $\tilde \calC(R+kr) \setminus {\calC}(R+(k+1)r)$ with $0 \leq k < R/r$ are realised. 
	We conclude that there exists a universal constant $c> 0$ such that, for every $0 \leq k < R/r$ 
	\begin{align*}
		\P_{\lambda_c}\big[\calC(R + k r)\big] - \P_{\lambda_c}\big[{\calC}(R+(k+1)r)\big] & \\
		= \P_{\lambda_c}\big[\tilde \calC(R + k r) \setminus {\calC}(R+(k+1)r)\big] &\geq c \P_{\lambda_c}\big[\tilde \calC(R) \setminus {\calC}(R+r)\big].
	\end{align*}
	Sum now the above over $k$ to deduce that 
	\begin{align*}
		1 \geq \P_{\lambda_c}\big[\calC(R)\big] - \P_{\lambda_c}\big[{\calC}(2R)\big] 
		\geq c \tfrac{R}r \P_{\lambda_c}\big[\tilde \calC(R) \setminus {\calC}(R+r)\big].
	\end{align*}
	Apply~\eqref{eq:rRR2} to obtain the desired inequality. 
\end{proof}

\subsection{Near-critical percolation}\label{sec:near_critical}


A fundamental question around the phase transition of percolation is the speed at which the model increases from its sub-critical behaviour to its super-critical one as $\lambda$ increases. 
In infinite volume the transition occurs instantaneously, but if we limit ourselves to a finite window, say the square of side-length $L$, 
then the model exhibits critical behaviour for an interval of intensities $\lambda$ around $\lambda_c$ called the {\em critical window}. 

Alternatively, one may state that, for any given $\lambda \neq \lambda_c$, the model behaves critically at scales up to some $L(\lambda)$, and sub- or super-critically above this scale. The scale $L(\lambda)$ is called the {\em characteristic length}, and may be shown to be equivalent to the better-known correlation length. 

This phenomenon was first proven by Kesten in 1987 \cite{Kesten1987} for Bernoulli percolation on planar lattices, 
along with an asymptotic expression for $L(\lambda)$ in terms for the number of pivotals in a box of size $n$.
Kesten's study of the near critical regime produced the so-called scaling relations, which link the algebraic decays of different natural observables of the critical and near-critical models. See also \cite{Nolin2008} for a more modern exposition of Kesten's result, and \cite{DCMT21} for a alternative proof. 

The principle of universality suggests that Kesten's results extend to a large variety of percolation models in the plane. 
They were indeed proven for Voronoi percolation in~\cite{Van19},
and appropriate alterations of Kesten's relations were extended to FK-percolation in \cite{DCM20}.

We claim that Kesten's result also applies to our model of continuum percolation. 
Below, we state a consequence of the more general theory of near-critical percolation designed to assist us in the proof of our main results. 
For $n\geq 1$, write 
\begin{align}\label{eq:alpha}
	\alpha_n := \frac{1}{ \pi_4(n) n^2}. 
\end{align}
As illustrated by the next theorem, $\alpha_n$ is the size of the {\em critical window} at scale $n$:
that is, it is the amount by which the critical parameter should be perturbed to observe off-critical features in a box of size $n$. 

\begin{theorem}[Crossings in near-critical percolation]\label{thm:nc_crossings}
	For any $\delta > 0$ there exist positive constants 
	$c(\delta),C(\delta) > 0$ such that, for all $n \geq 1$:
	\begin{align}
		&\P_{\lambda_c - C(\delta) \alpha_n}[\cross(n,2n) ] \leq \delta,\qquad
		\P_{\lambda_c + C(\delta) \alpha_n}[\cross(2n,n) ] \geq 1 - \delta,\quad \text{ and } \label{eq:nc_crossings1}\\
		&\big|\P_{\lambda}[\cross(n) ]  - \P_{\lambda_c}[\cross(n) ] \big|\leq \delta 
		\qquad \text{ when $|\lambda - \lambda_c| <  c(\delta)\alpha_n$}.\label{eq:nc_crossings2}
	\end{align}
\end{theorem}

The technique developed in \cite{Kesten1987} is easily adapted to continuum percolation when the radii of the discs are fixed (as is the case here), or have compact support in $(0,+\infty)$. Beyond these situations, conditions on the tails of the distribution of radii towards $0$ and $\infty$ are necessary, and the adaptation would require significant additional work. 


In the rest of this section, we overview the classical argument of Kesten and discuss how it needs to be adapted to continuum percolation in order for it to produce Theorem~\ref{thm:nc_crossings}. We start by sketching  Kesten's argument in our context. 
\medskip

For $\delta > 0$ define the \textit{characteristic length} at $\lambda>0$ as
\begin{equation}\label{def:cor:length} 
    \begin{split}
     L_\delta(\lambda):=
    \begin{cases} 
     	\inf\{n\geq 1 : \P_\lambda\big(\mathrm{cross}(n)\big)\geq  1-\delta\}&\quad\text{if } \lambda> \lambda_c,\\
    	\inf\{n\geq 1 : \P_\lambda\big(\mathrm{cross}^{*}(n)\big)\geq 1-\delta\}& \quad\text{if } \lambda< \lambda_c.
      \end{cases} 
    \end{split}
\end{equation}
Fix $\delta$ for the rest of the section and omit it from the notation. 
Below we use the notation $\asymp$ to relate two quantities whose ratios are uniformly bounded, with constants that may depend on $\delta$. 

A first step in the proof of Theorem~\ref{thm:nc_crossings} is to observe that the RSW theory implies uniform bounds on crossing probabilities ``under the characteristic length''. 
Indeed, Proposition~\ref{prop:RSW} implies that for any $\rho > 0$ there exists $c = c(\rho) > 0$ such that 
\begin{align}\label{eq:BXP}
	c < \P_\lambda(\mathrm{cross}(\rho n,n)) < 1- c \quad \text{ for all $\lambda$ and $1 \leq n \leq L(\lambda)$}.
\end{align}
This fact will be used implicitly in all arguments below. 

Observe now that Russo's formula~\eqref{eq:russoP:der} applied to the event $\mathrm{cross}(n)$ reads 
\begin{equation}\label{eq:Russo2}
	\frac{\mathrm{d}}{\mathrm{d}\lambda}\P_{\lambda}[\mathrm{cross}(n)] = \int_{[-n,n]^2} \P_\lambda \left[\mathrm{Piv}_{x}(\mathrm{cross}(n))\right] dx.
\end{equation}
For points $x$ in the bulk of $[-n,n]^2$, that is at a distance of order $n$ from the boundary, 
the probability of pivotality may be approximated by that of the four arm event. 
A slightly involved analysis shows that the bulk provides the major contribution to~\eqref{eq:Russo2} when $1 \leq n \leq L(\lambda)$. Thus we find 
\begin{equation}\label{eq:Russo3}
	\frac{\mathrm{d}}{\mathrm{d}\lambda}\P_{\lambda}[\mathrm{cross}(n)] \asymp n^2\, \P_{\lambda}[A_4(1,n)] \qquad \text{ for all $1 \leq n\leq L(\lambda)$}. 
\end{equation}

A similar reasoning may be used to bound the logarithmic derivative of $\P_{\lambda}[A_4(1,n)]$ by $c_0  n^2\, \P_{\lambda}[A_4(1,n)]$, for some universal constant $c_0$. 
Integrating both of these expressions, we conclude that 
\begin{equation*}
	\Big|\log \frac{ \P_{\lambda}[A_4(1,n)]}{ \P_{\lambda_c}[A_4(1,n)]} \Big|
	\leq c_0( \P_{\lambda}[\mathrm{cross}(n)] - \P_{\lambda_c}[\mathrm{cross}(n)]), \quad  \text{ for any $\lambda > 0$ and  $1 \leq n \leq L(\lambda)$}.
\end{equation*}
Now, since the right-hand side of the above is contained in $[-1,1]$, we conclude that 
\begin{align}\label{eq:stability_of_four_arm}
	\P_{\lambda}[A_4(1,n)] \asymp \P_{\lambda_c}[A_4(1,n)] \qquad  \text{ for any $\lambda > 0$ and  $1 \leq n \leq L(\lambda)$}.
\end{align}
This result is known as the {\em stability of arm-event probabilities} within the critical window.
It is the crucial step in Kesten's argument \cite{Kesten1987} and in its extensions \cite{Van19,DCM20}; 
a different proof of~\eqref{eq:stability_of_four_arm} is the object of \cite{DCMT21}.

Finally, plugging~\eqref{eq:stability_of_four_arm} back into~\eqref{eq:Russo3} and integrating, we find 
\begin{align*}
	\P_{\lambda}[\mathrm{cross}(n)] - \P_{\lambda_c}[\mathrm{cross}(n)] \asymp (\lambda - \lambda_c)\, n^2 \pi_4(n)
	\quad   \text{ for any $\lambda > 0$ and  $1 \leq n \leq L(\lambda)$}.
\end{align*}
The above directly implies Theorem~\ref{thm:nc_crossings}.\medskip

The program described above applies to continuum percolation with only slight additions. 
When proving~\eqref{eq:Russo2} for the logarithmic derivative of the four arm event, 
the argument showing that the bulk provides the majority of the contribution uses the existence of $c > 0$ such that 
\begin{align}\label{eq:four_arm_bound}
	\P_{\lambda}[A_4(r,R)] \geq c (r/R)^{2-c} \qquad  \text{ for any $\lambda > 0$ and $r < R \leq L(\lambda)$},
\end{align}
which may colloquially be stated as ``the four arm exponent is strictly smaller than two inside the critical window''.

In order to show~\eqref{eq:four_arm_bound}, Kesten compares the four arm event to the five arm one, which he shows has associated exponent equal to $2$. 
For continuum percolation, an appropriate definition of the five-arm event is necessary for this reasoning to function. 
For $r <R$, let $A_5(r,R)$ be the event that there exist five disjoint paths $\gamma_1,\dots, \gamma_5$ 
in $\overline{B(0,R)}\setminus B(0,r)$, each starting on $\partial B(0,r)$ and ending on $\partial B(0,R)$, 
distributed in counter-clockwise order, with $\gamma_1,\gamma_3, \gamma_5 \in \mathcal O$ and $\gamma_2,\gamma_4 \in \mathcal V$
and such that there exist two disjoint families of discs in $\mathcal O$, the first covering $\gamma_1$ and the second covering $\gamma_5$.
The discs of the two families may overlap, but no disc is allowed to belong to both families. 

Using this definition, the same arguments as for Bernoulli percolation on the square lattice \cite{Kesten1987,Nolin2008} show the existence of $c > 0$ such that
\begin{align*}
	\P_{\lambda}[A_5(r,R)] \geq c (r/R)^{2} 
	\quad\text{ and }\quad 
	\P_{\lambda}[A_4(r,R)] \geq (R/r)^{c}\P_{\lambda}[A_5(r,R)] 
\end{align*}
for any  $\lambda > 0$ and $r < R \leq L(\lambda)$.
This suffices to prove~\eqref{eq:four_arm_bound}, and therefore Theorem~\ref{thm:nc_crossings}. 



\section{Different couplings: a key observation}\label{sec:key}

For a point process configuration  $\eta$ and $r \geq 0$, set:
\begin{equation}\label{eq:enlarged_occ}
	\OO^{(r)}= \bigcup_{x\in \mathrm{supp}(\eta)}B(x,r).
\end{equation}
With this notation, $\OO =\OO^{(1)}$.
Write $\OO^{(r)} \in \cross(n)$ for the event that $\OO^{(r)}$ contains a horizontal crossing of $[-n,n]^2$. 
More generally, write $\OO^{(r)} \in \cross(m,n)$ for the event that $\OO^{(r)}$ contains a path crossing $[-m,m] \times [-n,n]$ horizontally. 
Finally, set $\V^{(r)} = \R^2 \setminus \OO^{(r)}$ and write $\V^{(r)} \in \cross^*(n)$ for the event 
$\V^{(r)}$ contains a horizontal crossing of $[-n,n]^2$. 

The key to our argument is contained in the following two simple observations. 

\begin{lemma}\label{lem:rescale}
	For any $\lambda > 0$:
	\begin{align}\label{eq_vacant_width_enlargeBalls}
		\w_n^* =& \:2 \sup \{ \varepsilon \geq 0 : \V^{(1+\varepsilon)} \in \cross^*(n)\}\quad \text{ and }\\
\label{eq_ineq_occupied_width_enlargeBalls}
		\w_n \geq & \: 2 \sqrt{1 - \inf \{ r \leq 1 : \OO^{(r)} \in \cross(n+ \sqrt{1 - r^2},n-1)\}^2}.
		\end{align}
	where the supremum and infimum above are considered equal to $0$ and $1$, respectively, if the set in question is empty. 
\end{lemma}

\begin{remark}\label{rem:supercritical_occupied}
	The above provides only a lower bound on the width of occupied crossings. 
	This is for good reason, as the two quantities in~\eqref{eq_ineq_occupied_width_enlargeBalls} are not generally equal. 
	For $\lambda \leq \lambda_c$, it is expected that the two quantities are typically equal, but that fails for general values of $\lambda$. 
Indeed, for $\lambda$ sufficiency large, one may typically find $\w_n > 2$ due to the creation of crossings of $[-n,n]^2$ by ``double paths'' of discs.
\end{remark} 

\begin{figure}[h!]
    \begin{center}
    \includegraphics[scale=.98]{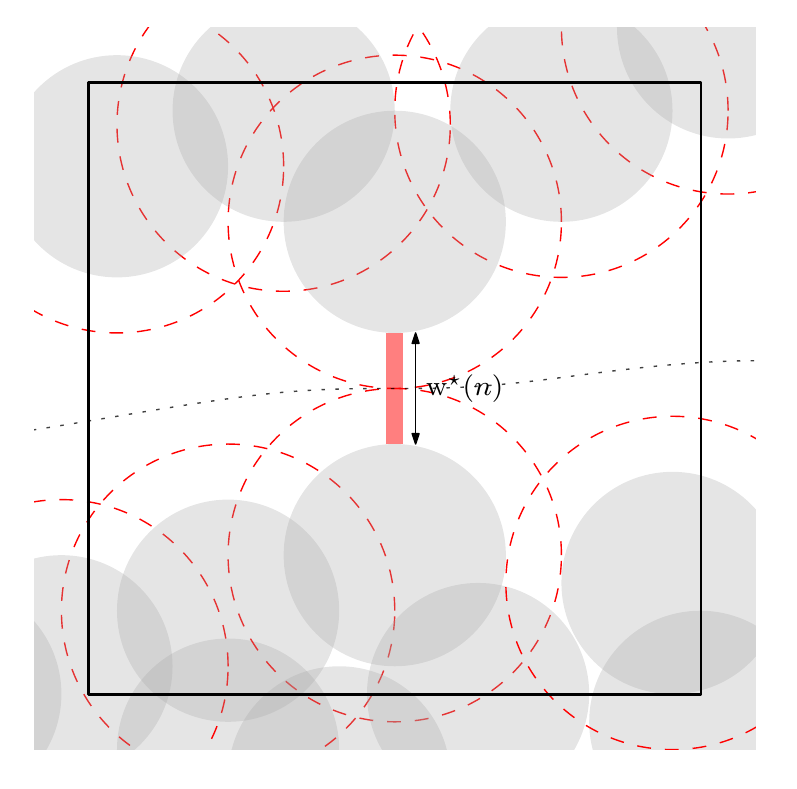}
    \caption{Computing the width of a vacant crossing by enlarging the balls.}
    \label{fig:vacant_cross}
    \end{center}
\end{figure}

\begin{proof}
	We start with~\eqref{eq_vacant_width_enlargeBalls}. 
	The equality is trivial when $\mathcal V \notin \cross^*(n)$, as both terms are equal to $0$. 
	Fix $\eta$ for which  $\mathcal V \in \cross^*(n)$. Figure~\ref{fig:vacant_cross} may be useful for understanding this proof. 

	Fix $0 < \varepsilon < \w^*(n)/2$ and let $\gamma$ be a horizontal crossing of $[-n,n]^2$ 
	for which $\textrm{dist}(\gamma, \OO) > \varepsilon$;
	the existence of such a path is guaranteed by the definition~\eqref{eq:def:width} of $\w_n^*$.
	Then $\gamma \in \V^{(1+\varepsilon)}$, and therefore $\V^{(1+\varepsilon)} \in \cross^*(n)$.
	This allow us to conclude that 
	\begin{align*}
		\w_n^* \leq 2\cdot \sup \{ \varepsilon \geq 0 \;:\; \V^{(1+\varepsilon)} \in \cross^*(n)\}.
	\end{align*}
	
	Conversely, for $\varepsilon$ such that $\V^{(1+\varepsilon)} \in \cross^*(n)$, let $\gamma$ be a horizontal crossing of $[-n,n]$ contained in $\V^{(1+\varepsilon)}$. Then $\textrm{dist}(\gamma, \OO) \geq \varepsilon$, and therefore $\w_n^* \geq 2 \varepsilon$. 
	This proves that 
	\begin{align*}
		\w_n^* \geq 2\cdot \sup \{ \varepsilon \geq 0 \;:\; \V^{(1+\varepsilon)} \in \cross^*(n)\}.
	\end{align*}
	Combine the last two displays to obtain the equality in~\eqref{eq_vacant_width_enlargeBalls}. \medskip 
	
	We move on to the inequality~\eqref{eq_ineq_occupied_width_enlargeBalls}, namely that involving $\w_n$. Figure~\ref{fig:occupied_cross_w} may be useful for understanding this proof. 
	The inequality is trivial when the right-hand side is equal to $0$. 
	We focus on the converse, and we fix $\eta$ and  $r  < 1$ such that $\mathcal O^{(r)} \in \cross(n + \sqrt{1 - r^2},n-1)$.
	Then there necessarily exists a family $x_0,\dots, x_k \in \eta \cap [-n,n]^2$ 
	such that $|x_i - x_{i-1}| \leq 2r$ for all $i = 1,\dots,k$, 
	and with $x_0$ and $x_k$ at a distance at least $1 - \sqrt{1 - r^2}$ from the left and right sides of $ [-n,n]^2$, respectively. 
	
	Consider now $\gamma$ to be the shortest path going through the vertices $x_0,\dots, x_k$ in this order. 
	Furthermore, add to $\gamma$ the initial horizontal segment going from the left side of $[-n,n]^2$ to $x_0$
	and the final horizontal segment going from $x_k$ to the right side of $[-n,n]^2$
	Then $\gamma$ crosses $[-n,n]^{2}$ horizontally. 
	
	Now, as may be observed in Figure~\ref{fig:occupied_cross_w}, 
	\begin{align*}
		\tfrac12 \w_n 
		\geq {\rm dist}(\gamma, \mathcal V) 
		\geq {\rm dist} \big(\gamma, \big[\bigcup_{i=0}^k B_1(x_i)\big]^c\big) 
		\geq \sqrt{1 - r^2}.
	\end{align*}
	Taking the infimum over $r$ in the above, we find~\eqref{eq_ineq_occupied_width_enlargeBalls}.
	
	\begin{figure}
	\begin{center}
	\includegraphics[width = .9\textwidth]{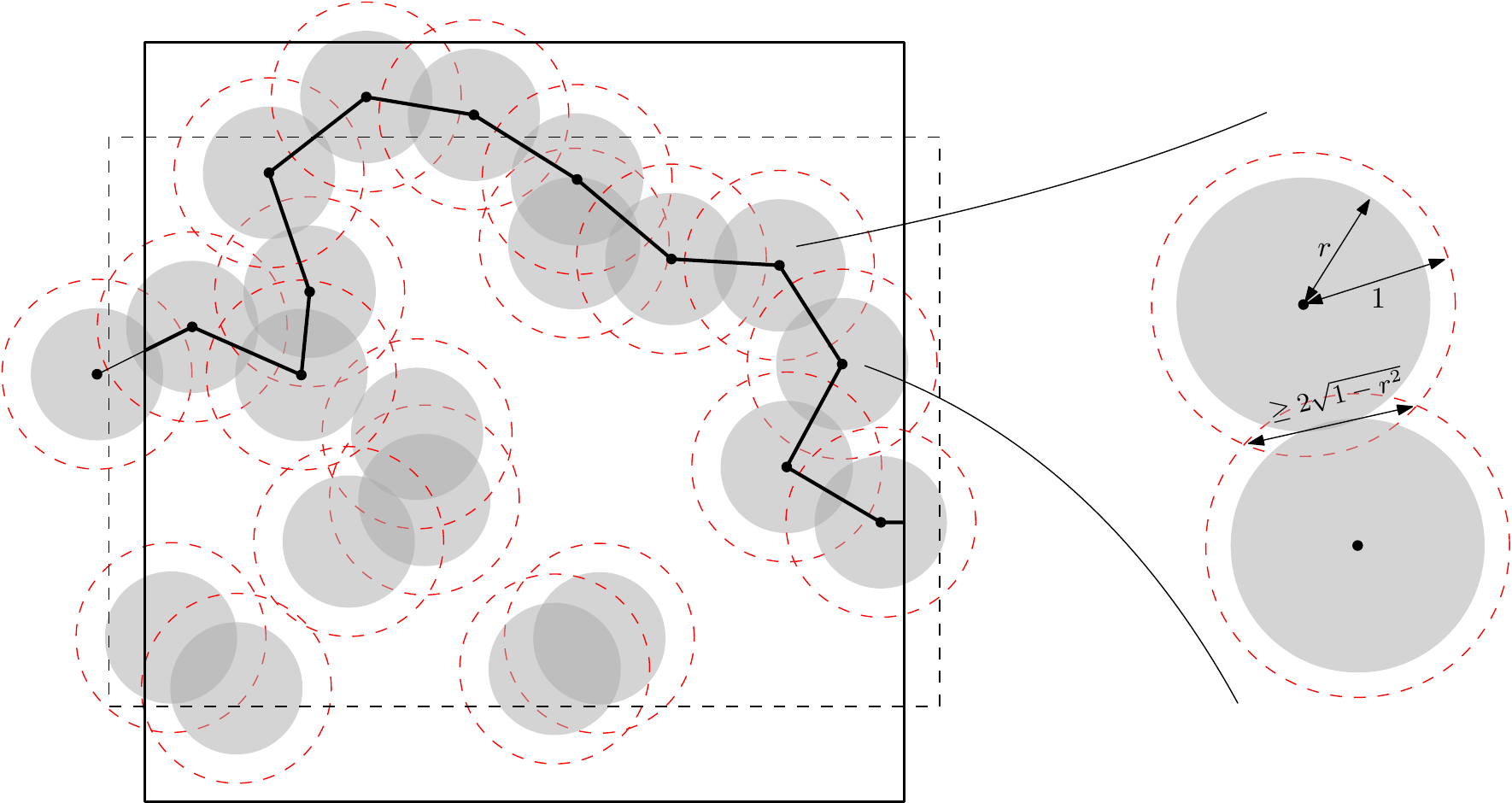}
	\caption{When $\mathcal O^{(r)} \in \cross(n+ \sqrt{1 - r^2},n-1)$, one may identify a chain of points of $\eta$, each at a distance at most $2r$ from the previous, contained in $\bbR \times [-n,n]$, with the first and last within a distance $r$ of the left and right sides of the rectangle, respectively. 
	The path $\gamma$ (bold black path) is obtained by interpolating linearly between these points, and potentially connecting the first and last points by horizontal lines to the sides of $[-n,n]^2$. 
	The distance from $\gamma$ to $\big[\bigcup_{i=0}^k B_1(x_i)\big]^c$ is attained at the center of one of the segments $[x_{i-1},x_i]$.}
	\label{fig:occupied_cross_w}
	\end{center}
	\end{figure}
\end{proof}

The second key observation is related to the scaling properties of Poisson point processes. 

\begin{lemma}\label{lem:rescale2}
	For every $\lambda>0$ and $r>0$, the law of $\frac{1}{r}\OO^{(r)}$ under $\P_\lambda$ is equal to the law of $\OO$ under $ \P_{\lambda r}$.
	In particular 
	\begin{align}\label{eq:rescale}
		\P_\lambda[\OO^{(r)} \in \cross(n)] = \P_{\lambda r}[\OO\in \mathrm{cross}(n/r)].
	\end{align}
\end{lemma}

\begin{proof}
	Fix $\eta$. Observe that 
	$$ \tfrac{1}{r}\OO^{(r)} =  \bigcup_{x\in \frac1r\mathrm{supp}(\eta)}B(x,1).$$
	Now, if $\eta$ is a Poisson point process of intensity $\lambda$, 
	then $\frac1r\mathrm{supp}(\eta)$ is distributed as a Poisson point process of intensity $\lambda r$. The result follows. 
\end{proof}

The two lemmas above combine to prove the following corollary.

\begin{corollary}\label{cor:w_distrib}
	For any $a \geq 0$ and $\lambda > 0$
	\begin{align} 
		\P_\lambda[\w_n^* \leq 2a ] &=  \P_{\lambda(1+a)}[\mathcal O \in \cross(\tfrac{n}{1+a})]  &&\text{ and } \label{eq:w_n^*_distrib}\\
		\P_\lambda[\w_n > 2a ] & \geq \P_{\lambda\sqrt{1 - a^2}}[\mathcal O \in \cross(\tfrac{n+a}{\sqrt{1 - a^2}}, \tfrac{n-1}{\sqrt{1 - a^2}}) ].  && \label{eq:w_n_distrib}
	\end{align}
\end{corollary}

\begin{proof}
	We start with~\eqref{eq:w_n^*_distrib}. Observe that, due to~\eqref{eq_vacant_width_enlargeBalls}, for any $a > 0$, 
	\begin{align*}
		\P_\lambda[\w_n^* \leq 2a ] 
		=\P_\lambda[\mathcal V^{(1+a)} \notin \cross^*(n) ] 
		=\P_\lambda[\mathcal O^{(1+a)} \in \cross(n) ] 
		=\P_{\lambda(1+a)}[\mathcal O \in \cross(\tfrac{n}{1+a}) ].
	\end{align*}
	The second equality is due to the duality property and the invariance under rotation by $\pi/2$; 
	the last equality is due to Lemma~\ref{lem:rescale2}. \medskip 
	
	We now turn to~\eqref{eq:w_n_distrib}. By~\eqref{eq_ineq_occupied_width_enlargeBalls} we have that:
	$$
		\{\w_n\geq 2a \} 
		\supset \big\{\OO^{(\sqrt{1-a^2})} \in \cross(n+a,n-1) \big\}
	$$
	Therefore, employing Lemma~\ref{lem:rescale2} we find:
	\begin{align*}
		\P_\lambda[\w_n \geq 2a ] 
		\geq \P_\lambda[\mathcal O^{(\sqrt{1 - a^2})} \in \cross(n+a,n-1) ] 
		=\P_{\lambda\sqrt{1 - a^2}}[\mathcal O \in \cross(\tfrac{n+a}{\sqrt{1 - a^2}}, \tfrac{n-1}{\sqrt{1 - a^2}}) ].
	\end{align*}
\end{proof}

\section{Proof of Theorem~\ref{thm:main_vacant} and of the lower bounds in Theorem~\ref{thm:main_occupied}}\label{sec:proofs_easy}

The proofs of  Theorem~\ref{thm:main_vacant} and of the lower bounds in Theorem~\ref{thm:main_occupied} 
are easy consequences of Corollary~\ref{cor:w_distrib}. 
As Corollary~\ref{cor:w_distrib} or Lemma~\ref{lem:rescale} make give no upper bounds on $\w_n$, 
the upper bounds in~\eqref{eq:main_critical_primal} and~\eqref{eq:main_supercritical_primal} will be harder to prove. They are postponed to the next section. \\
Recall the notation~\eqref{eq:alpha} $\alpha_n =  \tfrac{1}{n^2\pi_4(n)}$.

\begin{proof}[Proof of Theorem~\ref{thm:main_vacant}]
	Let us start with the {\bf critical} case, $\lambda = \lambda_c$. 
	Fix $\delta > 0$.
	Then, for $C >c > 0$ and $n\geq 1$, due to~\eqref{eq:w_n^*_distrib}, 
	\begin{align*}
		&1 - \P_{\lambda_c}\big[2c \alpha_n \leq \w^*_n \leq 2C \alpha_n \,\big|\, \cross^*(n) \big]\\
		&\quad  = \tfrac{1}{\P_{\lambda_c}[\cross^*(n)]}\Big(
		        \P_{\lambda_c}[ \w^*_n <2c \alpha_n]  - \P_{\lambda_c}[ \w_n  = 0]  + \P_{\lambda_c}[ \w^*_n >2C \alpha_n] \Big)\\
		&\quad  \leq  C_1
		\Big(\P_{\lambda_c(1 +c \alpha_n) }\big[\cross(\tfrac{n}{1+c \alpha_n})\big] -\P_{\lambda_c}[\cross(n)] 
		+ 1 - \P_{\lambda_c( 1 +C \alpha_n) }\big[\cross(\tfrac{n}{1+C \alpha_n})\big] \Big)
	\end{align*}
	where $C_1$ is a universal constant provided by Proposition~\ref{prop:RSW}. 
	We used Lemma~\ref{lem:rescale2} to relate the bounds on $\w_n^*$ to crossing events. 
	Applying Theorem~\ref{thm:nc_crossings}, we deduce that $c$ and $C$ may be chosen small and large enough, respectively, independently of $n$, such that the above is bounded as
	\begin{align*}
	1 - \P_{\lambda_c}\big[2c \alpha_n \leq \w^*_n \leq 2C \alpha_n \,\big|\, \cross^*(n) \big]
		&\leq \delta +  C_1 \big(\P_{\lambda_c}\big[\cross(\tfrac{n}{1+c \alpha_n})\big] -\P_{\lambda_c}[\cross(n)] \big).
	\end{align*}
	Finally, Lemma~\ref{lem:rRR} proves that the last term above may also be rendered smaller than $\delta$, provided that $c$ is small enough.
	\medskip 
	
	\noindent We now consider the {\bf subcritical} case $\lambda < \lambda_c$. As in the critical case, applying Corollary~\ref{cor:w_distrib} yields, for any $n \geq 1$ and $C >0$,
    \begin{align}
	   &\P_{\lambda}\Big[\big| \w^*_n - 2\big(\tfrac{\lambda_c}{\lambda} - 1\big)\big| > 2C\alpha_n \;\Big|\; \cross^*(n) \Big] \nonumber\\
	   & \qquad \leq \tfrac{1}{\P_\lambda[\cross^*(n)]}\Big(
	   \P_{\lambda}\Big[ \w_n^* <  2 \big(\tfrac{\lambda_c - \lambda}{\lambda} - C\alpha_n\big) \Big] +
	   \P_{\lambda}\Big[ \w_n^* >  2 \big(\tfrac{\lambda_c - \lambda}{\lambda} + C\alpha_n\big) \Big]
	   \Big)\nonumber\\
	    & \qquad \leq  \tfrac{1}{\P_\lambda[\cross^*(n)]}\left(
	   \P_{\lambda_c - C \lambda \alpha_n}\big[ \cross(n_-)) \big]+
	   1 - \P_{\lambda_c + C \lambda \alpha_n}\big[ \cross(n_+) \big] \right),
	   \label{eq:proofvacant1}
	\end{align}
	where $n_\pm = {n}/\big({1 + \tfrac{\lambda_c - \lambda}{\lambda} \pm C\alpha_n}\big)$. 
	Now, due to~\eqref{eq:nc_crossings1}, $C$ may be chosen large enough (depending on $\lambda$, but not on $n$) so that
	\begin{align*}
		   \P_{\lambda_c - C \lambda \alpha_n}\big[ \cross(n_-)) \big] \leq \delta  \quad \text{ and }\quad
		   1 - \P_{\lambda_c + C \lambda \alpha_n}\big[ \cross(n_+) \big] \leq \delta,
	\end{align*}
	 for all $n$ large enough.
	Finally, for all $n$ sufficiently large, the pre-factor in the right-hand side of~\eqref{eq:proofvacant1} is smaller than $2$. 
	Combining the above inequalities leads to the desired conclusion. 
	\medskip 
	
	\noindent Finally, let us consider the {\bf supercritical} case~\eqref{eq:main_supercritical_dual}. 
	We will treat separately the lower and upper bounds on $\w^*_n$. We start with the former.
	
	Due to~\eqref{eq:w_n^*_distrib}, for any $c > 0$ and $n\geq 1$, we have 
	\begin{align}\label{eq:mscd0}
		&\P_{\lambda}\big[\w^*_n < \tfrac{2c}{n} \;\big|\; \cross^*(n) \big]
		 = \tfrac{1}{\P_\lambda[\cross^*(n)]}\big(  \P_{\lambda(1+c/n)}[\cross(\tfrac{n}{1+c/n})]  - \P_\lambda[\cross(n)]\big) \\
		& \quad  \leq \tfrac{1}{\P_\lambda[\cross(n)^c]}\big(  \P_{\lambda(1+c/n)}[\cross(n-c,n) \setminus \cross(n)] +
		\P_{\lambda(1+c/n)}[\cross(n)] - \P_\lambda[\cross(n)] \big),\nonumber
	\end{align}
	where the inequality is obtained by adding and subtracting $\P_{\lambda(1+c/n)}[\cross(n)]$ and using the inclusion of rectangles. 
	We will bound separately the first term and the difference of the last two terms appearing in the parentheses above. We start with the latter. 
	
	Consider the measure $P$ which consists in choosing a Poission process $\eta$ of intensity $\lambda$ 
	and an independent additional Poisson point process $\tilde\eta$ of intensity $c \lambda /n$. 
	Write $\calO$ and $\tilde\calO$ for the occupied sets produced by these two processes. 
	Then 
	\begin{align}\label{eq:mscd12}
		  \frac{\P_{\lambda(1+c/n)}[\cross(n)] - \P_\lambda[\cross(n)]}{\P_\lambda[\cross(n)^c]} = 
		P \big[\calO \cup \tilde\calO \in \cross(n)\,\big|\, \calO \notin\cross(n)  \big].
	\end{align}
	
	Now, when $\calO \notin \cross(n)$, write $\calC$ for the union of the vacant clusters crossing $[-n,n]^2$ vertically. 
	Also write $\calC^{(1)} = \{x + z\,:\, x \in \calC,\, |z| \leq 1\}$ for the fattening of $\calC$ by $1$ 
	and  $\calA(\calC^{(1)})$ for the area covered by $\calC^{(1)}$.

	Since $\lambda > \lambda_c$, the conditioning on $\calO \notin \cross(n)$ is very degenerate, which renders the typical cluster $\calC$ very thin. 
	Indeed, a straightforward adaptation of the theory of \cite{CamIof02}
	induces the existence of a constant $C(\delta)> 0$ such that 
	\begin{align}\label{eq:OZarea}
		P \big[\calA(\calC^{(1)}) \geq  C(\delta) n \,\big|\, \calO \notin\cross(n)\big] < \delta, \qquad \text{ for all $n$}.
	\end{align}
	See Figure~\ref{fig:supercritOZ} (left) for an illustration. 
	Due to the independence of $\calO$ and $\tilde\calO$, the probability that a disc of $\tilde\calO$ intersects $\calC$ may be computed as 
	\begin{align*}
		P [\tilde \calO \cap \calC \neq \emptyset \,|\,  \calO] = \exp[- c \lambda\, \calA(\calC^{(1)})/n].
	\end{align*}
	Fix $c > 0$ sufficiently small that  $\exp(- c \lambda C(\delta)) > 1- \delta$, with $C(\delta)$ the constant given by~\eqref{eq:OZarea}. 
	Then we find that 
	\begin{align}
		&P \big[\calO \cup \tilde\calO \in \cross(n) \,\big|\,\calO \notin\cross(n)  \big]\nonumber\\
		&\qquad \leq P \big[ \calA(\calC^{(1)}) \geq  C(\delta) n \,\big|\, \calO \notin\cross(n)  \big]
		+	P\big[\calC \cap \tilde\calO\neq \emptyset \,\big|\, \calO \notin\cross(n),\, \calA(\calC^{(1)}) <  C(\delta) n  \big]
		\nonumber\\
		&\qquad < 2\delta.\nonumber
	\end{align}
	Combined with~\eqref{eq:mscd12}, the above yields
	\begin{align}\label{eq:mscd1}
		 \tfrac{1}{\P_\lambda[\cross(n)^c]}\big(\P_{\lambda(1+c/n)}[\cross(n)] - \P_\lambda[\cross(n)] \big)
		 \leq 2\delta.
	\end{align}
	
	We now turn to $\P_{\lambda(1+c/n)}[\cross(n-c,n) \setminus \cross(n)]$. 
	For this event to occur, there needs to exists a vacant vertical crossing of $[-n,n]^2$ 
	that visits $[-n,-n + c] \times [-n,n]$ or $[n- c,n] \times [-n,n]$. 
	Since any such crossing is essentially straight (i.e. has width $o(n)$), we find that, for $n$ large enough 
	\begin{align*}
	\P_{\lambda(1+c/n)}[\cross(n-c,n)\,|\, \cross(n)^c] \leq \delta.
	\end{align*}
	Thus,
	\begin{align}\label{eq:mscd2}
	    \P_{\lambda(1+c/n)}[\cross(n-c,n) \setminus \cross(n)]
		\leq \delta\, \P_{\lambda(1+c/n)}[\cross(n)^c]
		\leq \delta \,\P_{\lambda}[\cross(n)^c],
	\end{align}
	where the second inequality comes from the monotonicity in $\lambda$.
	The term $\delta$ in the first inequality could even be replaced $O(1/n)$ by studying more carefully the vacant cluster, 
	but this is unnecessary for our purposes. 
	
	Inserting~\eqref{eq:mscd1} and~\eqref{eq:mscd2} in~\eqref{eq:mscd0} we find that 
	\begin{align}\label{eq:mscd3}
		\P_{\lambda}\big[\w^*_n < \tfrac{2c}{n} \;\big|\; \cross^*(n) \big]
		\leq 3\delta,
	\end{align}
	which is the desired lower bound on $\w^*_n$. 
	
	\begin{figure}
	\begin{center}
	\includegraphics[width=0.4\textwidth, page = 2]{fig/supercritOZ.pdf}\hspace{.1\textwidth}
	\includegraphics[width=0.4\textwidth, page = 1]{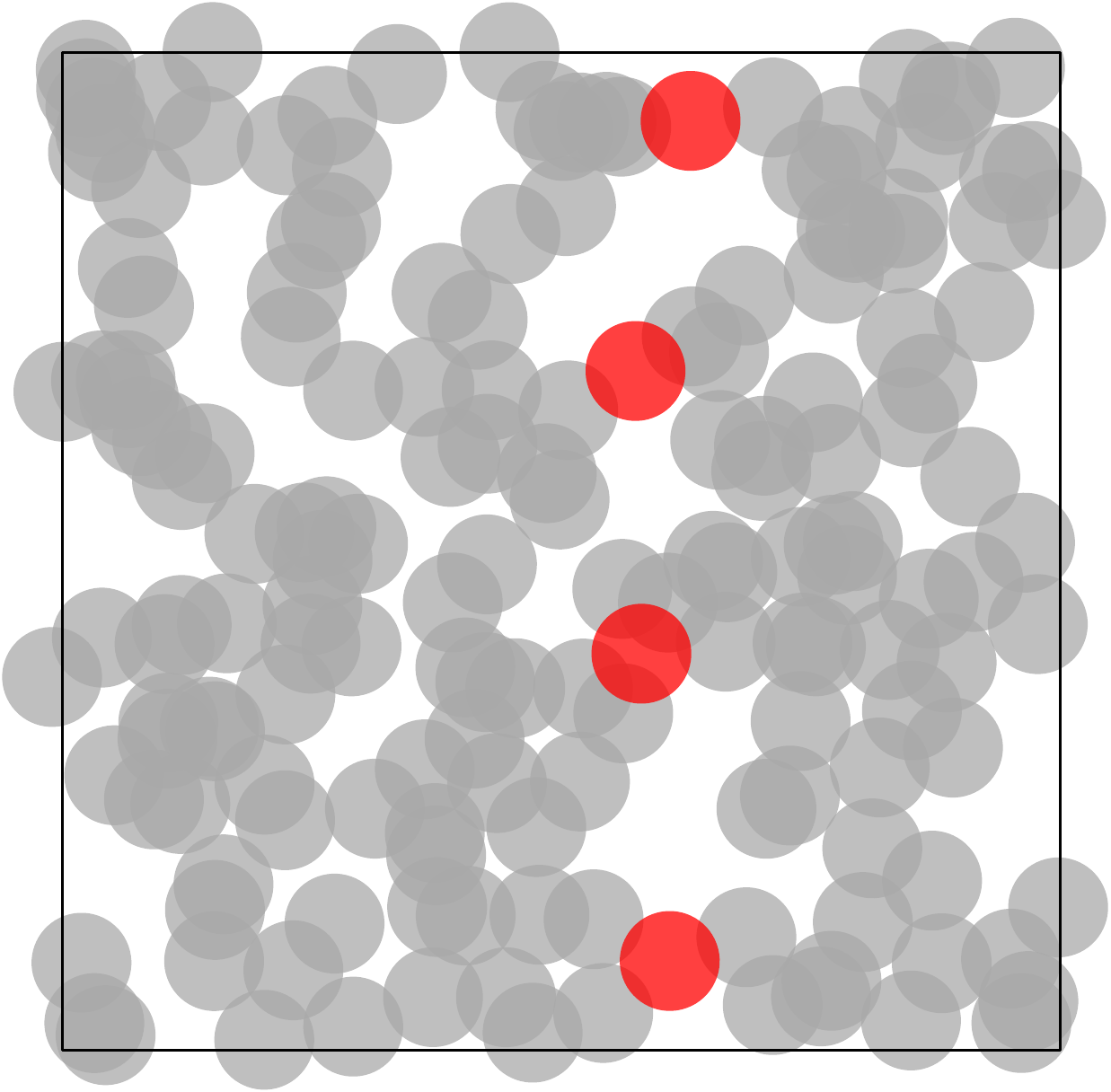}	
	\caption{{\bf Left:} in the supercritical regime, when  $\cross(n)$ fails, the vacant cluster crossing $[-n,n]^2$ vertically is thin; 
	it typically has an area $\calA(\calC)$ of order $n$. 
	{\bf Right:} in the same situation, there exists a linear number of places where adding one disk induces an occupied horizontal crossing. The centres of these potential disks form $\Pi$.}
	\label{fig:supercritOZ}
	\end{center}
	\end{figure}

	We now turn to the upper bound on $\w^*_n$. 
	Using~\eqref{eq:w_n^*_distrib}, for $C>0$ and $n\geq 1$, 
	\begin{align*}
		\P_{\lambda}\big[\w^*_n >  \tfrac{2C}{n} \;\big|\; \cross^*(n) \big]
		 &= \tfrac{1}{\P_\lambda[\cross^*(n)]}\big(1 - \P_{\lambda(1+C/n)}[\cross(\tfrac{n}{1+C/n})]\big) \\
		 &\leq \tfrac{1}{\P_\lambda[\cross(n)^c]} \P_{\lambda(1+C/n)}[\cross(n, n-C)^c]
	\end{align*}
	The second inequality is due to the inclusion of rectangles.
	As in~\eqref{eq:mscd2},
	\begin{align*}
		 \P_{\lambda(1+C/n)}[\cross(n, n-C)^c]  \leq \P_{\lambda(1+C/n)}[\cross(n)^c] (1 + \delta). 
	\end{align*}
	for $n$ large enough. 
	Thus,
	\begin{align}\label{eq:mscd5}
		\P_{\lambda}\big[\w^*_n >  \tfrac{2C}{n} \;\big|\; \cross^*(n) \big]
		\leq(1 + \delta) \frac{\P_{\lambda(1+C/n)}[\cross(n)^c] }{\P_\lambda[\cross(n)^c]}.
	\end{align}

	Using the same notation $P$, $\calO$ and $\tilde \calO$ as above, with $\tilde \eta$ having intensity $C/n$, we find 
	\begin{align}\label{eq:mscd6}
		\P_{\lambda}[\cross(n)^c] - \P_{\lambda(1+C/n)}[\cross(n)^c] 
		= P[\calO \notin \cross(n) \text{ but } \calO \cup\tilde\calO \in \cross(n) ].
	\end{align}
	When $\calO \notin \cross(n)$, \cite{CamIof02} states that there typically exists a unique, thin vacant cluster connecting the top and bottom of $\cross(n)$, 
	and that this cluster has a linear number of pivotals. 
	Let us give a precise statement of this fact.  
	For a configuration with $\calO \notin \cross(n)$, 
	write $\Pi$ for the set of pivotal points, that is $\Pi := \{x \in \bbR^2 \,:\, \calO \cup B_1(x) \in \cross(n)\}$;
	see Figure~\ref{fig:supercritOZ} (right) for an illustration. 
	Then an adaptation of \cite{CamIof02} to the continuous setting shows the existence of a constant $c(\delta) > 0$ such that
	\begin{align*}
		\P_\lambda[\calA(\Pi) > c(\delta) n \,|\,\calO \notin \cross(n)] \geq 1 - \delta,
	\end{align*}
	where $\calA(\Pi)$ denotes the area of $\Pi$. 
	It follows that 
	\begin{align*}
		&P[\calO \cup\tilde\calO \in \cross(n) \,|\,\calO \notin \cross(n)]\\
		&\qquad \geq P[\calO \cup\tilde\calO \in \cross(n) \,|\,\calO \notin \cross(n)\text{ and } \calA(\Pi) > c(\delta) n]
		\,P[\calA(\Pi) > c(\delta) n \,|\,\calO \notin \cross(n)]\\
		&\qquad \geq (1 - e^{- c(\delta) C \lambda})(1-\delta),
	\end{align*}
	since $ 1 - e^{- c(\delta) C \lambda}$ bounds from below the probability that $\tilde\eta$ contains a pivotal point. 
	Taking $C > 0$ large enough, depending on $c(\delta)$ and $\lambda$ but not on $n$, we conclude that the above is larger than $1- 2\delta$ for all $n$ large enough. 
	Inserting this into~\eqref{eq:mscd6}, we find 
	\begin{align*}
	 	\P_{\lambda(1+C/n)}[\cross(n)^c]  \leq 2\delta \, \P_{\lambda}[\cross(n)^c].
	\end{align*}
	The above together with~\eqref{eq:mscd5} imply that 
	\begin{align}\label{eq:mscd7}
		\P_{\lambda}\big[\w^*_n >  \tfrac{2C}{n} \;\big|\; \cross^*(n) \big]
		 &\leq  (1 + \delta)2\delta.
	\end{align}	
	Finally,~\eqref{eq:mscd3} and~\eqref{eq:mscd7} yield~\eqref{eq:main_supercritical_dual}.
\end{proof}

We now turn to the results concerning the maximal width in the occupied set. 
In this section, we only prove lower bounds for $\w_n$; upper bounds are proved in the next section. 

\begin{proof}[Proof of Theorem~\ref{thm:main_occupied}; lower bounds]
	Fix $\delta > 0$. We will prove in each case that, by taking $c$ small enough, 
	$\P_{\lambda}[\w_n \leq c \,\theta(n) \,|\, \cross(n) ]$ may be rendered smaller than some explicit function of $\delta$, which tends to $0$ as $\delta \to 0$. The threshold $\theta(n)$ depends on whether $\lambda$ is smaller, equal or larger than $\lambda_c$, and is announced in Theorem~\ref{thm:main_occupied}. 
	\medskip

\noindent
	We start with the {\bf critical} case~\eqref{eq:main_critical_primal}.
	Then, for $c > 0$ and $n\geq 1$, due to~\eqref{eq:w_n_distrib}, 
	\begin{align}
		&P_{\lambda_c}\big[\w_n \leq 2\sqrt {c \alpha_n} \,\big|\, \cross(n) \big]\nonumber\\
		&\qquad= \tfrac{1}{\P_{\lambda_c}[\cross(n)]}\big(\P_{\lambda_c}[\cross(n)] - \P_{\lambda_c}[\w_n > 2\sqrt {c \alpha_n}]\big) \nonumber\\
		&\qquad\leq \tfrac{1}{\P_{\lambda_c}[\cross(n)]}\big(\P_{\lambda_c}[\cross(n)] - 
		\P_{\lambda_c\sqrt{1 - c \alpha_n}}[\cross(\tfrac{n+\sqrt {c \alpha_n}}{\sqrt{1 - c \alpha_n}}, \tfrac{n-1}{\sqrt{1 - c \alpha_n}}) ]\big)\nonumber\\
		&\qquad\leq \tfrac{1}{\P_{\lambda_c}[\cross(n)]}\big(\P_{\lambda_c}[\cross(n)] - 
		\P_{\lambda_c (1 - c \alpha_n)}[ \cross(n + 3 c n \alpha_n,  n-1) ]\big),
		\label{eq:mcp2}
	\end{align}
	provided that $n$ is large enough such that $(1 -c \alpha_n)^{-1/2} > 1 + c\alpha_n$ and $ n \sqrt {c \alpha_n} \geq 1$.
	The last inequality uses the monotonicities of $\P_\lambda(\cross(a,b))$ in $\lambda$, $a$ and $b$. 
	
	Now, due to Theorem~\ref{thm:nc_crossings}, we deduce that $c$ may be chosen such that, for all $n$, 
	\begin{align*}
		\P_{\lambda_c (1 - c \alpha_n)}[\cross(n + 3 c \alpha_n,  n-1) ]
		&\geq \P_{\lambda_c}[ \cross(n + 3 c n \alpha_n,  n-1) ] + \delta\\
		&\geq \P_{\lambda_c}[ \cross(n) ] + 2\delta,
	\end{align*}
	with the second inequality due to Lemma~\ref{lem:rRR}, $n$ large enough (depending on $c$), and we have used the a-priori estimates on the 4-arms event. 
	
	Combining the above with~\eqref{eq:mcp2}, and using the RSW inequality~\eqref{eq:RSW}, we conclude that, 
	for $c$ small enough and all $n$ larger than some threshold, 
	\begin{align}
		\P_{\lambda_c}\big[\w_n \leq 2\sqrt {c \alpha_n} \,\big|\, \cross(n) \big] \leq C_0 \delta,		\label{eq:mcp3}
	\end{align}
	where $C_0$ is a universal constant. 
\medskip

\noindent We continue with the {\bf supercritical} case~\eqref{eq:main_supercritical_primal}.
	Fix $\lambda > \lambda_c$. 
	As in the critical case, applying ~\eqref{eq:w_n_distrib} yields, for $C>0$ and all $n$ large enough, 
    \begin{align*}
	     &\P_{\lambda}\Big[\w_n \geq 2 \sqrt{1 - \big(\tfrac{\lambda_c}{\lambda} + \tfrac{C}{\lambda}\alpha_n\big)^2}\;\Big|\; \cross(n) \Big] \\
	     &\qquad\geq \frac{1}{ \P_{\lambda}[ \cross(n)] }
	      \P_{\lambda_c + C\alpha_n }\Big[\cross\Big(\tfrac{\lambda\big(n + \sqrt{1 - \big(\tfrac{\lambda_c}{\lambda} + \tfrac{C}{\lambda}\alpha_n\big)^2}\big)}{\lambda_c+ C\alpha_n},\tfrac{\lambda(n - 1)}{\lambda_c+ C\alpha_n}\Big) \Big]     \\
	     &\qquad\geq  \frac{1}{ \P_{\lambda}[ \cross(n)] }
	      \P_{\lambda_c + C\alpha_n }[ \cross(2c(\lambda)n, c(\lambda) n)],
	\end{align*}
	where $c(\lambda)$ is a constant depending only on $\lambda$.
	Theorem~\ref{thm:nc_crossings} states that $C$ may be chosen large enough so that the second term in the right-hand side of the above is larger than $1-\delta$ for all $n$. 
	The first term in the right hand-side above converges to $1$ as $n\to\infty$ due to the choice of $\lambda$ being supercritical. 
	Thus, for $C$ chosen sufficiently large and all $n$ large enough
    \begin{align*}
	 	\P_{\lambda}\Big[\w_n \geq 2 \sqrt{1 - \big(\tfrac{\lambda_c}{\lambda} + \tfrac{C}{\lambda}\alpha_n\big)^2}\;\Big|\; \cross(n) \Big] \geq 1 - 2\delta,
	\end{align*}	
	as claimed. 
	\medskip 
	
	\noindent
	Finally, let us analyse the {\bf subcritical} case~\eqref{eq:main_subcritical_primal}.
	The strategy is the same as for the supercritical case~\eqref{eq:main_supercritical_dual} for the vacant set. 
	Fix $\lambda  < \lambda_c$. 
	Due to~\eqref{eq:w_n_distrib}, for any $c > 0$ and $n\geq 1$, we have 
	\begin{align}\label{eq:mscp0}
		\P_{\lambda}\big[\w_n < 2 \sqrt{\tfrac{c}{n}} \;\big|\; \cross(n) \big]
		& = \tfrac{1}{\P_\lambda[\cross(n)]}\big( \P_\lambda[\cross(n)] - \P_{\lambda(1-\frac{c}n)}[\cross(\tfrac{n+\sqrt{c/n}}{\sqrt{1-\frac{c}n}}, \tfrac{n-1}{\sqrt{1-\frac{c}n}})]\big) \nonumber\\
		&   \leq \tfrac{1}{\P_\lambda[\cross(n)]}\Big(  
		 \P_\lambda[\cross(n)] - \P_{\lambda(1-\frac{c}n)}[\cross(n)]\nonumber\\
		&  
		+  \P_{\lambda(1-\frac{c}n)}[\cross(n)] -  \P_{\lambda(1-\frac{c}n)}\big[\cross\big(\tfrac{n+\sqrt{\frac{c}n}}{\sqrt{1-\frac{c}n}}, \tfrac{n-1}{\sqrt{1-\frac{c}n}}\big)\big] \Big).
	\end{align}
	We will bound separately the two differences appearing in the parentheses above, starting with the first one. 
	
	Consider the measure $P$ which consists in choosing a Poission process $\eta$ of intensity $\lambda (1-c/n)$ 
	and an independent additional Poisson point process $\tilde\eta$ of intensity $c \lambda /n$. 
	Write $\calO$ and $\tilde\calO$ for the occupied sets produced by these two processes. 
	Then 
	\begin{align}\label{eq:mscp12}
		\P_\lambda[\cross(n)] - \P_{\lambda(1-c/n)}[\cross(n)]= 
		P \big[\calO \notin\cross(n) \text{ but } \calO \cup \tilde\calO \in \cross(n)  \big].
	\end{align}
	
	Now, when $\calO \cup \tilde\calO \in \cross(n)$, write $\calC$ for the union of the occupied clusters crossing $[-n,n]^2$ horizontally. 
	Since $\lambda < \lambda_c$, $\calC$ is typically formed of a single, thin cluster. 
	In particular it is of linear ``volume'', both in area and in the number of discs belonging to it. 
	Indeed, a straightforward adaptation of the theory of \cite{CamIof02}
	induces the existence of a constant $C(\delta)> 0$ such that 
	\begin{align}\label{eq:mscp13}
		\P_\lambda \big[|(\eta\cup \tilde \eta) \cap \calC| \geq  C(\delta) n \,\big|\, \calO\cup  \tilde\calO  \in\cross(n)\big] < \delta, \qquad \text{ for all $n$}.
	\end{align}
	
	Now, under $P$ and conditionally on $\eta \cup \tilde \eta$, each point of $\eta \cup \tilde \eta$
	belongs to $\tilde \eta$ with a probability $ c/n$. Thus, whenever the event in~\eqref{eq:mscp13} occurs, 
	\begin{align}\label{eq:mscp14}
		P [\tilde\eta \cap \calC = \emptyset \,|\,  \eta \cup\tilde\eta] \geq \big(1-\tfrac{c}n\big)^{ C(\delta) n} \geq 1-\delta, 
	\end{align}
	provided that $c > 0$ is chosen sufficiently small. 
	Inserting~\eqref{eq:mscp12} and~\eqref{eq:mscp13} in~\eqref{eq:mscd12}, we find 
	\begin{align}\label{eq:mscp1}
		\P_{\lambda(1-c/n)}[\cross(n)] \geq (1-2\delta)\, \P_\lambda[\cross(n)].
	\end{align}
	
	We now turn to the second difference in~\eqref{eq:mscp0}.
	Assuming that $c>0$ is sufficiently small and $n$ sufficiently large, we have 
	\begin{align*}
		\P_{\lambda(1-c/n)}\Big[\cross\Big(\tfrac{n+\sqrt{c/n}}{\sqrt{1-c/n}}, \tfrac{n-1}{\sqrt{1-c/n}}\Big)\Big]
		&\geq \P_{\lambda(1-c/n)}[\cross(n + 2c, n-1)]\\
		&\geq (1 - \delta)\P_{\lambda(1-c/n)}[\cross(n + 2c, n) ]\\
		&\geq (1 - 2\delta)\P_{\lambda(1-c/n)}[\cross(n)].
	\end{align*}
	The first inequality is due to the inclusion of rectangles and some basic algebra.
	The second is obtained in the same way as~\eqref{eq:mscd2} and is valid for $n$ large enough: 
	the occupied component producing $\cross(n + 2c, n-1)$ avoids approaching the top and bottom sides of the rectangles with high probability. 
	The third is valid for $c$ small enough (independent of $n$) and is based on the same reasoning as Lemma~\ref{lem:rRR}, namely that 
	\begin{align*}
		&\P_{\lambda(1-c/n)}[\cross(n + 2(k+1)c, n) \setminus \cross(n + 2kc, n) ] \\
		&\qquad\qquad\qquad \geq c_0 \P_{\lambda(1-c/n)}[\cross(n + 2c, n) \setminus \cross(n) ]\qquad \qquad  \text{ for all $k \leq 1/c$}, 
	\end{align*}
	where $c_0 > 0$ is some constant depending only on $\lambda$, not on $n$ or $c$. 
	Finally, the above combined with~\eqref{eq:mscp1} shows that
	\begin{align}\label{eq:mscp23}
		\P_{\lambda(1-c/n)}[\cross(n)] -\P_{\lambda(1-c/n)}\Big[\cross\Big(\tfrac{n+\sqrt{c/n}}{\sqrt{1-c/n}}, \tfrac{n-1}{\sqrt{1-c/n}}\Big)\Big]
		&\leq 2\delta \P_{\lambda(1-c/n)}[\cross(n)] \nonumber \\
		&\leq 2\delta \P_{\lambda}[\cross(n)] 
	\end{align}
	where the second inequality comes form the monotonicity in $\lambda$.
	Putting~\eqref{eq:mscp1} and~\eqref{eq:mscp23} together, we conclude that
	\begin{align}\label{eq:mscp3}
	\P_{\lambda}\big[\w_n < 2 \sqrt{{c}/{n}} \;\big|\; \cross(n) \big] \leq 4\delta
	\end{align}
	as claimed. 
\end{proof}

\section{Remaining proofs}\label{sec:proofs_pivotals}

In this section we prove the upper bounds~\eqref{eq:main_subcritical_primal} and~\eqref{eq:main_critical_primal} on $\w_n$ in the subcritical and critical cases. These could not be proved using the techniques of the previous section due to the missing upper bound on $\w_n$ in~\eqref{eq:w_n_distrib}. 

The method presented here could most likely be used to obtain all the results announced in Theorems~\ref{thm:main_vacant} and~\ref{thm:main_occupied}, 
but is less elegant than that of Section~\ref{sec:proofs_easy} and would require significantly more work. 

The arguments in this section use some fine properties of critical and sub-critical percolation, which we will state explicitly, but not prove. 
Proofs are available in the literature for Bernoulli percolation on the square lattice (reference will be provided), and these may be adapted directly to our setting. 
We start with a series of definitions.

	For $x \in \bbR^2$, write $\Lambda(x) = [-2,2]^2 + x$ for the square of side-length $4$ centred at the point $x$. 
	For $n\geq1$ and points $x_1,\dots, x_k \in \mathbb R^2$, we say that $(x_1,\dots,x_k)$ is a pivotal chain for $\cross(n)$ if 
	$\calO \setminus \bigcup_{i = 1}^k \Lambda(x_i) \notin \cross(n)$, but the set is minimal for this property, in that
	for any $X \subsetneq \{1,\dots,k\}$, $\calO \setminus \bigcup_{i \in X} \Lambda(x_i) \in \cross(n)$.
	
 	Call $0$ a \textit{thin point} if there exist exactly two discs with centres in $[-4,4]^2 \setminus \Lambda(0)$ 
	and if these discs have centres in $[-3.5,-2.5]\times [-1,1]$ and $[2.5,3.5]\times [-1,1]$, respectively. 
	Call these centres $\ell(0)$ and $r(0)$, respectively. 
	See Figure~\ref{fig:thin} for an illustration. 	
	
	Notice that the property that $0$ is thin imposes no restriction on the discs inside $\Lambda(0)$, nor on those outside of  $[-4,4]^2$. 
	In particular, there may exist more than two discs intersecting $[-4,4]^2 \setminus \Lambda(0)$. 
	However, no disc centred inside $\Lambda(0)$ can intersect discs centred outside $[-4,4]^2$.
	
	We call $0$ a thin connected point if the two discs with centres in $[-4,4]^2 \setminus \Lambda(0)$ are connected by the occupied set formed of the discs with centres in $\Lambda(0)$.
	Write $\w(0)$ for the maximal width of the occupied connection between $\ell(0)$ and $r(0)$ produced by the discs in $\Lambda(0)$. 
	Set $\w(0) = 0$ if no such connection exists.
	The following lemma is a simple but key observation. 
	
	\begin{lemma}\label{lem:connect_thin}
		Fix $\lambda > 0$. There exists $c_0 = c_0(\lambda) > 0$ depending on $\lambda$ such that, for any $a \in [0,1]$,
		\begin{align*}
			\bbP_\lambda[0< \w(0) < 2a \,|\, x \text{ thin and $\eta$ outside $\Lambda(0)$}] \geq c_0 a^2.
		\end{align*}
	\end{lemma}
	
	In particular, applying the above with $a =1$ shows that thin points are connected with positive probability. 
	
	\begin{figure}
	\begin{center}
	\includegraphics[width = 0.45\textwidth, page =1]{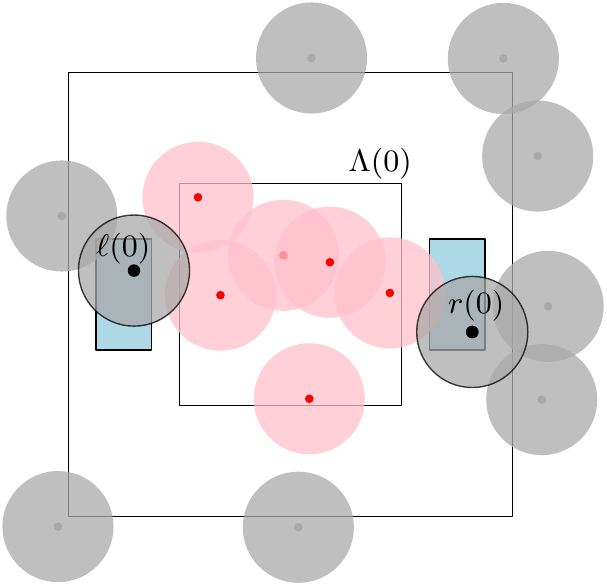}
	\includegraphics[width = 0.45\textwidth, page =2]{fig/thin.pdf}
	\caption{
	{\bf Left:} A thin connected point. 
	The only two discs with centres in $[-4,4]^2 \setminus \Lambda(0)$ are marked in bold; 
	their centres belong to the two blue regions on the side of $\Lambda(0)$. 
	They are connected by the pink discs with centres in $\Lambda(0)$.
	{\bf Right:} For $0< \w(0) < 2a$, it suffices to have no point of $\eta$ in the hashed part of $\Lambda(0)$ and
	a point of $\eta$ in the blue region, which is then connected to $r(0)$ by other discs centred in $\Lambda(0)$.
	The blue region has area of order $a^2$. }
	\label{fig:thin}
	\end{center}
	\end{figure}
	
	\begin{proof}
		The proof follows from a simple geometrical construction. 
		For $\{0< \w(0) < 2a\}$ to occur, 
		it suffices that there exists a disk in $\Lambda(0)$ at a distance between $1- a^2$ and $1$ from $\ell(0)$ which is connected to $r(0)$ by discs inside $\Lambda(0)$, 
		and that there exists no other point of $\eta$ in $\Lambda(0)$ at a distance at most $1$ from $\ell(0)$; see Figure~\ref{fig:thin}. 
		The existence of the first point occurs with a probability at least $c_0 a^2$ (for some positive constant $c_0$ that depends on $\lambda$); 
		all other conditions are satisfied with positive probability. 
	\end{proof}

	The definitions of thin and connected thin point apply by translation to any point $x \in \bbR^2$. 
	Write $\ell(x)$, $r(x)$ and $\w(x)$ for the associated notions.
	
	For $n,m ,K \geq 1$, let $\PivChain_n(m,K)$ be the event that there exists $1 \leq k\leq K$ and 
	disjoin sets $X_1,\dots,X_k \subset (8\bbZ)^2$ such that 
	\begin{itemize}
	\item $|X_i| \geq m$ for each $i$, 
	\item any $x \in \bigcup_{i=1}^k X_i$ is thin and connected,
	\item for any $x_1 \in X_1$, \dots $x_k \in X_k$, $x_1,\dots, x_k$ is a pivotal chain for $\cross(n)$. 
	\end{itemize} 
	Observe that here the points of each set $X_i$ are required to be placed on a fixed lattice, at large distance from each other. 
	Also note that $\PivChain_n(m,K) \subset \cross(n)$. 
	See Figure~\ref{fig:pivotal_chains} for an illustration of $\PivChain_n(m,K)$.

	\begin{figure}
	\begin{center}
	\includegraphics[width = 0.6\textwidth, page =1]{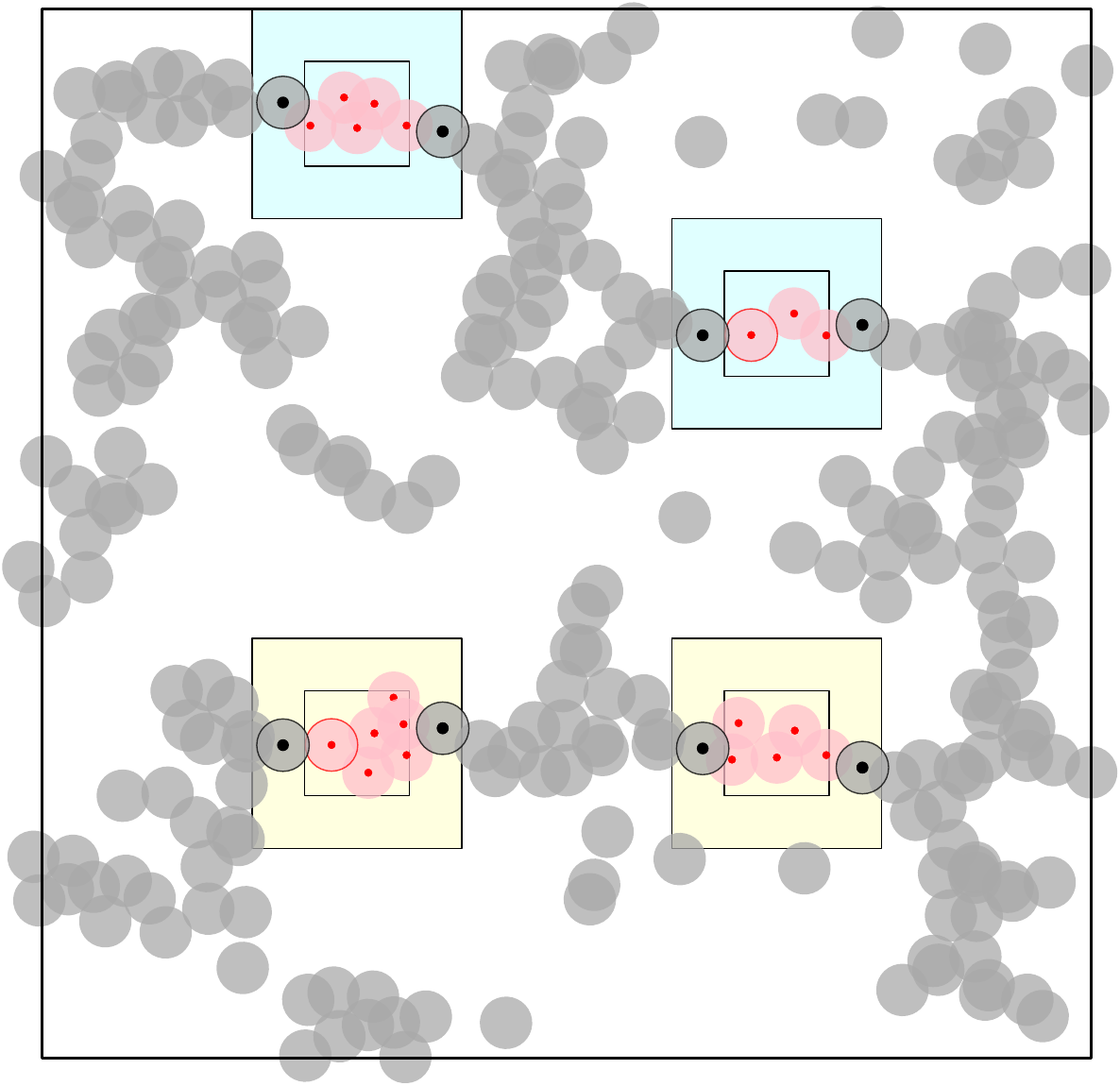}
	\caption{A situation in $\PivChain_n(m,K)$ with two sets $X_1$, $X_2$ (the corresponding boxes are blue and yellow, respectively) and two points $x$ in each set.
	All thin points are connected. The overall width of the crossing is small if there exists a point $x$ in each set $X_1$ and $X_2$ with $\w(x)$ small (see top right and bottom left boxes).}
	\label{fig:pivotal_chains}
	\end{center}
	\end{figure}
	
\subsection{Critical occupied case: upper bound}\label{sec:proofs_piv_crit}

	It is a property of critical percolation that, even when $\cross(n)$ occurs, there exists a chain of large vacant clusters that cross $[-n,n]^2$ vertically,
	with the clusters almost touching each other at many points. Moreover, these clusters are few in number and are all of large size. 
	The following lemma is a detailed such statement. 	
	
	\begin{lemma}\label{lem:piv_chain}
		For any $\delta > 0$, there exist $c > 0$ and $K \geq 1$ such that, for all $n$ large enough, 
		\begin{align*}
			\bbP\big[\PivChain_n(c n^2 \pi_4(n),K) \,\big|\, \cross(n)\big] \geq 1 - \delta. 
		\end{align*}
	\end{lemma}

	The lemma is a consequence of the RSW theorem and of the a-priori bounds on the four and six-arm probabilities, which may be proven similarly to the Bernoulli case. 
	Lemma~\ref{lem:piv_chain} may be proved in the same way as~\cite[Thm~7.5]{DCMV21}.
	One should mention that~\cite[Thm~7.5]{DCMV21} does produce pivotal chains, but not with points belonging to the lattice $(8\bbZ)^2$, nor points that are thin. 
	To obtain thin points aligned to the lattice, one needs to use the separation of arms for the four arm event. This is a tedious but standard approach which we will not detail. 
	
\begin{proof}[Proof of Theorem~\ref{thm:main_occupied}, critical case~\eqref{eq:main_critical_primal}]
	Recall that the lower bound on $\w_n$ was proved in~\eqref{eq:mcp3}. We will focus here on the upper bound. 
	Fix $\delta >0$ and let $K \geq 1$ and $c > 0$ be the constants provided by Lemma~\ref{lem:piv_chain}.
	Let $n$ be large enough that Lemma~\ref{lem:piv_chain} applies. 
	
	When $\PivChain_n(c n^2 \pi_4(n),K)$ occurs, let $\calX= (X_1,\dots, X_k)$ be the first family of sets of $(8\bbZ)^2$ 
	that satisfies the properties of $\PivChain_n(c n^2 \pi_4(n),K)$	according to some arbitrary order. 
	The properties of $\calX$ impose that any occupied path crossing $[-n,n]^2$ horizontally crosses all $[\Lambda(x)]_{x \in X_i}$ for some $i \leq k$. 
	As such, we find 
	\begin{align*}
		\w_n \leq \max_{1 \leq i \leq k} \min_{x \in X_i} \w(x). 
	\end{align*}
	
	Observe that the thin points of $(8\bbZ)^2$ may be determined by knowing $\eta$ outside of $\bigcup_{x \in (8\bbZ)^2}\Lambda(x)$. 
	Furthermore, for any family $(X_1,\dots, X_k)$ of sets of $(8\bbZ)^2$, 
	one may determine if it satisfies the properties of $\PivChain_n(c n^2 \pi_4(n),K) $ 
	by knowing $\eta$ outside of $\bigcup_{x \in (8\bbZ)^2}\Lambda(x)$, inside all $\Lambda(x)$ for points $x\in (8\bbZ)^2$ which are not thin, 
	and by knowing which thin points of $(8\bbZ)^2$ are connected. 

	It follows that for any possible realisation $\calX_0$ of $\calX$, 
	the law of $\eta$ knowing $\calX = \calX_0$ and the process $\eta$ outside of $\Lambda(\calX_0) := \bigcup_{x \in \calX_0} \Lambda(x)$
	is simply that of a Poisson point processes on $\Lambda(\calX_0)$ conditioned on each $x \in \bigcup_i X_i$ being connected. 
	In particular, Lemma~\ref{lem:connect_thin} shows that, for any $x \in \bigcup_i X_i$,
	\begin{align}\label{eq:www}
		\bbP_{\lambda_c}\big[\w(x) < 2a \,\big|\, \calX = \calX_0 \text{ and $\eta$ outside $\Lambda(\calX_0) $}\big] \geq c_0 a^2 \qquad  \text{for all $a \in [0,1]$}.
	\end{align}
	Apply this to $a = \sqrt{C \alpha_n}$ for some large constant $C$ to deduce that, for each $i$, 
	\begin{align*}
		\bbP_{\lambda_c}\big[\min_{x \in X_i } \w(x) \geq 2\sqrt{C \alpha_n} \,\big|\, \calX = \calX_0 \text{ and $\eta$ outside $\Lambda(\calX_0) $}\big] 
		< (1 - c_0 C \alpha_n)^{c/\alpha_n} 
		< \tfrac{\delta}K. 
	\end{align*}
	The first inequality is a direct consequence of~\eqref{eq:www}, 
	the fact that the restrictions of $\eta$ to the different $[\Lambda(x)]_{x \in X_i}$ are independent
	and that $|X_i| \geq c/\alpha_n$.  
	The second inequality is ensured by taking $C$ sufficiently large (depending on $c$ and $c_0$, but not on $n$). 
	We conclude from the above that
	\begin{align}\label{eq:www2}
		\bbP_{\lambda_c}\big[\max_{1\leq i \leq k} \min_{x \in X_i } \w(x) \geq \sqrt{C \alpha_n} \,\big|\, \calX = \calX_0 \text{ and $\eta$ outside $\Lambda(\calX_0) $}\big] 
		< \delta. 
	\end{align}
	
	Apply now~\eqref{eq:www2} to all realisations $\calX$ producing $\PivChain_n(c n^2 \pi_4(n),K)$, then integrate to obtain
	\begin{align*}
		&\bbP_{\lambda_c}[\w_n \geq 2\sqrt{C \alpha_n}] \\
		\qquad&\leq \bbP_{\lambda_c}[\w_n \geq 2\sqrt{C \alpha_n} \text{ and }\PivChain_n(c n^2 \pi_4(n),K)] 
		+\bbP_{\lambda_c}[\cross(n) \setminus \PivChain_n(c n^2 \pi_4(n),K) ] \\
		\qquad&\leq 2\delta,
	\end{align*}
	with the second inequality also due to Lemma~\ref{lem:piv_chain}. 
	Together with the upper bound~\eqref{eq:mcp3} already proved, this implies~\eqref{eq:main_critical_primal}.
\end{proof}

\subsection{Subcritical occupied case: upper bound}

In subcritical percolation $\cross(n)$ is very unlikely to occur. 
Moreover, when it does, the occupied cluster crossing $[-n,n]^2$ horizontally is very thin and contains a linear number of pivotals,
of which a linear number will be connected thin points. 
A straightforward adaptation of~\cite{CamIof02} yields the following statement. 

\begin{lemma}\label{lem:piv_chain_sc}
	Fix $\lambda < \lambda_c$. There exists $c_1 = c_1(\lambda) > 0$ depending only on $\lambda$ such that, 
	\begin{align}\label{eq:piv_chain_sc}
		\bbP_{\lambda} \big[\PivChain_n(c_1 n,1)\,\big|\, \cross(n)\big] \geq 1 - e^{-c_1 n}.
	\end{align}
\end{lemma}

\begin{proof}[Proof of Theorem~\ref{thm:main_occupied} subcritical case~\eqref{eq:main_subcritical_primal}]
	Recall that the lower bound on $\w_n$ was proved in~\eqref{eq:mscp3}. We will focus here on the upper bound. 
	Fix $\lambda <\lambda_c$, $\delta >0$ and let $c_1 > 0$ be the constant provided by Lemma~\ref{lem:piv_chain_sc}.
	Let $n$ be large enough so that $e^{-c_1 n} < \delta$. 
	
	The proof is similar to that of Section~\ref{sec:proofs_piv_crit}.
	When $\PivChain_n(c_1 n,1)$ occurs, let $\calX$ be the maximal set of $(8\bbZ)^2$ that satisfies the properties of $\PivChain_n(c_1n,1)$.
	Since we are considering situations where each point of $\calX$ is pivotal, one may define such a maximal set 
	(this was not the case in Section~\ref{sec:proofs_piv_crit}, where pivotal chains were considered). 
	Then, since any occupied path crossing $[-n,n]^2$ horizontally crosses all $[\Lambda(x)]_{x \in \calX}$, we find 
	\begin{align*}
		\w_n \leq \min_{x \in \calX} \w(x). 
	\end{align*}
	
	The same argument as in Section~\ref{sec:proofs_piv_crit} shows that, for any potential realisation $\calX_0$ of $\calX$, 
	the law of $\eta$ knowing $\calX = \calX_0$ and the process $\eta$ outside of $\Lambda(\calX_0) := \bigcup_{x \in \calX_0} \Lambda(x)$
	is simply that of a Poisson point processes on $\Lambda(\calX_0)$ conditioned on each $x \in \bigcup_i X_i$ being connected. 
	In particular, Lemma~\ref{lem:connect_thin} shows that, for any $x \in \bigcup_iX_i$,
	\begin{align}\label{eq:www12}
		\bbP_{\lambda}\big[\w(x) < 2a \,\big|\, \calX = \calX_0 \text{ and $\eta$ outside $\Lambda(\calX_0) $}\big] \geq c_0 a^2 \qquad  \text{for all $a \in [0,1]$}.
	\end{align}
	Apply this to $a = \sqrt{C/n}$ for some large constant $C$ to deduce that, for any $\calX_0$
	\begin{align}
		\bbP_{\lambda}\big[\min_{x \in \calX } \w(x) \geq 2\sqrt{C/n} \,\big|\, \calX = \calX_0 \text{ and $\eta$ outside $\Lambda(\calX_0) $}\big] 
		< (1 - c_0 C/n)^{c_1 n} 
		< \delta. 
		\label{eq:www22}
	\end{align}
	The second inequality is ensured by taking $C$ sufficiently large (depending on $c_0$ and $c_1$, but not on $n$). 
	
	Apply now~\eqref{eq:www22} to all realisations $\calX$ producing $\PivChain_n(c_1 n,1) $, then integrate to obtain
	\begin{align*}
		\bbP_{\lambda}[\w_n \geq 2\sqrt{C /n}] 
		&\leq \bbP_{\lambda}[\w_n \geq 2\sqrt{C/n} \text{ and }\PivChain_n(c_1 n,1) ] 
		+\bbP_{\lambda}[\cross(n) \setminus \PivChain_n(c_1 n,1) ] \\
		&\leq 2\delta\bbP_{\lambda}[\cross(n)],
	\end{align*}
	with the second inequality also due to Lemma~\ref{lem:piv_chain_sc} and the choice of $n$ large enough. 
	Together with the upper bound~\eqref{eq:mscp3} already proved, this implies~\eqref{eq:main_subcritical_primal}.
\end{proof}

\section{Questions}\label{sec:open_question}

In closing, let us discuss some related open questions. 

The most natural question is probably to extend the results beyond non-compactly supported radii distributions. 
The results surely fail when the tails of the distribution of radii are too heavy, but for quickly decaying distributions they should remain valid. 

The second question that comes to mind, is whether this analysis may be performed for randomly placed sets of any shape, rather than discs. 
For such sets, Corollary~\ref{cor:w_distrib}, which is the cornerstone of the proofs of Section~\ref{sec:proofs_easy}, ceases to hold. 
The method used in Section~\ref{sec:proofs_pivotals} and based on the study of pivotal points appears more robust, and may be used for general shapes. 
It would therefore be interesting to adapt this method to prove all results. 
Some problems may arise for lower bounds on $\w_n$ and $\w_n^*$, 
as the points where these minimal widths are reached are not always pivotals. 

A third question is related to the difference between the results for the occupied and vacant set. 
In the critical case, $\w_n$ is of the same order as $\sqrt{\w_n^*}$; the same phenomenon happens when comparing~\eqref{eq:main_subcritical_dual}
to~\eqref{eq:main_supercritical_primal} 
and~\eqref{eq:main_supercritical_dual} to~\eqref{eq:main_subcritical_primal}.
This difference appears to be due to the round shape of the discs, but a quick computation based on the method of Section~\ref{sec:proofs_pivotals} seems to suggest that the same is true when discs are replaced with squares. 
Is this phenomenon more general? 

Finally, note that the supercritical case of Theorem~\ref{thm:main_occupied} only offers a lower bound on $\w_n$. 
Indeed, the upper bound 
\begin{align*}
	\P_{\lambda}\Big[\w_n \leq 2 \sqrt{1 - \big(\tfrac{\lambda_c}{\lambda} - \tfrac{C}{\lambda n^2\pi_4(n)}\big)^2}\;\Big|\; \cross(n) \Big] &\geq  1-\delta,
\end{align*}
fails for $\lambda$ sufficiently large, due to the phenomenon explained in Remark~\ref{rem:supercritical_occupied}.
Still, one may expect it to hold for $\lambda$ sufficiently close to $\lambda_c$. Is this the case?

We close with a thought: this study appears to be specific to continuum percolation, with no apparent correspondence in the discrete. Is there one?

\bibliographystyle{amsplain}
\bibliography{bib}

\end{document}